\documentclass[12pt]{article}

\usepackage{amsmath, amssymb, booktabs, color, algorithm, algorithmic}
\usepackage{amsthm}
\usepackage{bm}
\usepackage{caption}
\usepackage{subcaption}
\usepackage{here}
\usepackage{comment}
\usepackage[dvipdfmx]{graphicx}

\newtheorem{thm}{Theorem}[section]

\newtheorem{lem}{Lemma}[section]

\newtheorem{prob}{Problem}[section]
\newtheorem{algo}{Algorithm}[section]

\newcommand{\argmin}{\operatornamewithlimits{argmin}}

\numberwithin{equation}{section}

\begin{document}
\makeatletter

\begin{center}
\large{\bf Line Search Fixed Point Algorithms Based on Nonlinear Conjugate Gradient Directions:
Application to Constrained Smooth Convex Optimization}\\
\small{This work was supported by the Japan Society for the Promotion of Science through a Grant-in-Aid for
Scientific Research (C) (15K04763).}
\end{center}\vspace{3mm}

\begin{center}
\textsc{Hideaki Iiduka}\\
Department of Computer Science, 
Meiji University,
1-1-1 Higashimita, Tama-ku, Kawasaki-shi, Kanagawa 214-8571 Japan. 
(iiduka@cs.meiji.ac.jp)
\end{center}

\vspace{2mm}

\footnotesize{
\noindent\begin{minipage}{14cm}
{\bf Abstract:}
This paper considers the fixed point problem for a nonexpansive mapping on a real Hilbert space and proposes novel line search fixed point algorithms to accelerate the search. The termination conditions for the line search are based on the well-known Wolfe conditions that are used to ensure the convergence and stability of unconstrained optimization algorithms. The directions to search for fixed points are generated by using the ideas of the steepest descent direction and conventional nonlinear conjugate gradient directions for unconstrained optimization. We perform convergence as well as convergence rate analyses on the algorithms for solving the fixed point problem under certain assumptions. The main contribution of this paper is to make a concrete response to an issue of constrained smooth convex optimization; that is, whether or not we can devise nonlinear conjugate gradient algorithms to solve constrained smooth convex optimization problems. We show that 
the proposed fixed point algorithms include ones with nonlinear conjugate gradient directions
which can solve constrained smooth convex optimization problems.
To illustrate the practicality of the algorithms, we apply them to concrete constrained smooth convex optimization problems, such as constrained quadratic programming problems and generalized convex feasibility problems, and numerically compare them with previous algorithms based on the Krasnosel'ski\u\i-Mann fixed point algorithm. The results show that the proposed algorithms dramatically reduce the running time and iterations needed to find optimal solutions to the concrete optimization problems compared with the previous algorithms.
\end{minipage}
 \\[5mm]


\hbox to14cm{\hrulefill}\par


\section{Introduction}\label{sec:1}
Consider the following {\em fixed point problem} (see \cite[Chapter 4]{b-c}, \cite[Chapter 3]{goebel1}, \cite[Chapter 1]{goebel2}, \cite[Chapter 3]{takahashi}):
\begin{align}\label{prob:1}
\text{Find } x^\star \in \mathrm{Fix}\left(T\right) := 
\left\{x^\star \in H \colon  T\left(x^\star\right) = x^\star \right\},
\end{align}
where $H$ stands for a real Hilbert space with inner product $\langle \cdot,\cdot\rangle$ and its induced norm $\| \cdot \|$,
$T$ is a {\em nonexpansive} mapping from $H$ into itself
(i.e., $\| T(x) - T(y) \| \leq \|x-y\|$ $(x,y\in H)$), and one assumes $\mathrm{Fix}(T) \neq \emptyset$.
Problem \eqref{prob:1} includes convex feasibility problems \cite{bau}, \cite[Example 5.21]{b-c}, 
constrained smooth convex optimization problems \cite[Proposition 4.2]{yamada},
problems of finding the zeros of monotone operators \cite[Proposition 23.38]{b-c}, and monotone variational inequalities \cite[Subchapter 25.5]{b-c}.

There are useful algorithms for solving Problem \eqref{prob:1}, such as the {\em Krasnosel'ski\u\i-Mann algorithm} \cite[Subchapter 5.2]{b-c}, \cite[Subchapter 1.2]{berinde}, \cite{cominetti2014,kra,mann}, {\em Halpern algorithm} \cite[Subchapter 1.2]{berinde}, \cite{halpern,wit}, and {\em hybrid method} \cite{nakajo2003} (Solodov and Svaiter \cite{solodov2000} proposed the hybrid method to solve problems of finding the zeros of monotone operators). 
This paper focuses on the Krasnosel'ski\u\i-Mann algorithm that has practical applications,
such as analyses of dynamic systems governed by maximal monotone operators \cite{bot2015} and nonsmooth convex variational signal recovery \cite{comb2007},
defined as follows: given the current iterate $x_n \in H$ and step size $\alpha_n \in [0,1]$, the next iterate $x_{n+1}$ of the algorithm is
\begin{align}\label{kra}
x_{n+1} :=  x_n +  \alpha_n \left( T \left(x_n \right) - x_n \right).
\end{align}
Assuming that $(\alpha_n)_{n\in\mathbb{N}}$ satisfies the condition,
\begin{align}\label{Dunn} 
\sum_{n=0}^{\infty} \alpha_n (1-\alpha_n) = \infty, 
\end{align}
the sequence $(x_n)_{n\in\mathbb{N}}$ generated by Algorithm \eqref{kra} weakly converges to a fixed point of $T$ (see, e.g., \cite[Theorem 5.14]{b-c}).
This result indicates that Algorithm \eqref{kra} with constant step sizes (e.g., $\alpha_n := \alpha \in (0,1)$ $(n\in\mathbb{N})$) or diminishing step sizes (e.g., $\alpha_n := 1/(n+1)$ $(n\in\mathbb{N})$) can solve Problem \eqref{prob:1}.
Propositions 10 and 11 in \cite{cominetti2014} indicate that Algorithm \eqref{kra} with Condition \eqref{Dunn}
has the following rate of convergence: for all $n\in \mathbb{N}$,
\begin{align}\label{rate_kra}
\left\|x_n - T \left(x_n\right) \right\| =O \left( \left\{ \sum_{k=0}^n \alpha_k \left(1-\alpha_k \right) \right\}^{-\frac{1}{2}}  \right)
\end{align}
(e.g., $\|x_n - T(x_n)\| = O(1/\sqrt{n+1})$ when $\alpha_n := \alpha \in (0,1)$ $(n\in\mathbb{N})$).
This fact implies that Algorithm \eqref{kra} with \eqref{Dunn} does not always have fast convergence and has motivated the development of modifications and variants 
for the Krasnosel'ski\u\i-Mann algorithm in order to accelerate Algorithm \eqref{kra}.

One approach to accelerate Algorithm \eqref{kra} with \eqref{Dunn} is to develop line search methods that can determine a more adequate step size than a step size satisfying \eqref{Dunn} at each iteration $n$ 
so that the value of $\|x_{n+1} - T(x_{n+1})\|$ decreases dramatically.
Magnanti and Perakis proposed an {\em adaptive line search framework} \cite[Section 2]{mag2004} that can determine step sizes to satisfy weaker conditions \cite[Assumptions A1 and A2]{mag2004} than \eqref{Dunn}. On the basis of this framework, they showed that Algorithm \eqref{kra}, with step sizes $\alpha_n$ satisfying the following {\em Armijo-type} condition, converges to a fixed point of $T$ \cite[Theorems 4 and 8]{mag2004}: given $x_n \in \mathbb{R}^N$, $\beta > 0, D > 0$, and $b\in (0,1)$, choose the smallest nonnegative integer $l_n$ so that $\alpha_n = b^{l_n}$ satisfies the condition,
\begin{align}\label{armijo0}
g_n\left(\alpha_n \right) - g_n\left(0 \right) \leq - D b^{l_n} \left\| T \left(x_n\right) - x_n \right\|^2,
\end{align}
where $g_n \colon [0,1] \to \mathbb{R}$ is a potential function \cite[Scheme IV]{mag2004} 
defined for all $\alpha \in [0,1]$ by
\begin{align}\label{gn}
\begin{split}
g_n \left(\alpha\right) 
&:= \| \left( x_n +  \alpha \left( T \left(x_n \right) - x_n \right) \right)
  - T \left( x_n +  \alpha \left( T \left(x_n \right) - x_n \right) \right)\|^2\\ 
&\quad  - \beta \alpha \left( 1 - \alpha \right) \|  T\left(x_n\right) - x_n  \|^2.
\end{split}
\end{align}
Theorem 5 in \cite{mag2004} shows that 
Algorithm \eqref{kra} with the Armijo-type condition \eqref{armijo0}
satisfies that $\| x_{n+1} - T(x_{n+1}) \|^2 \leq [1 - \beta (\alpha_n - 1/2)^2] \| x_n- T(x_n) \|^2$ $(n\in \mathbb{N})$,
which implies that the algorithm has that, for all $n\in \mathbb{N}$,
\begin{align}\label{rate_mag}
\left\|x_n - T \left(x_n\right) \right\| = O\left( \left\{ \sum_{k=0}^n \left(\alpha_k- \frac{1}{2} \right)^2 \right\}^{-\frac{1}{2}} \right).
\end{align}

In this paper, we introduce a line search framework using $P_n$ defined by \eqref{simple1}, \eqref{simple0}, and \eqref{simple}, which is the simplest of all potential functions including $g_n$ defined as in \eqref{gn}:
given $x_n, d_n \in H$, for all $\alpha \in [0,1]$,
\begin{align}
&x_n \left( \alpha  \right) := x_n + \alpha d_n, \label{simple1} \\
&Q_n \left(\alpha \right) := x_n \left(\alpha \right) - T \left( x_n \left( \alpha \right) \right), \label{simple0} \\
&P_n \left(\alpha \right) := \left\| Q_n \left(\alpha \right)  \right\|^2.\label{simple}
\end{align}
When $d_n := -(x_n - T (x_n))$ and $\alpha_n$ is given as in \eqref{Dunn}, 
the point $x_n(\alpha_n)$ in \eqref{simple1} coincides with $x_{n+1}$ defined by Algorithm \eqref{kra} with \eqref{Dunn}. 
Consider the following problem of minimizing $P_n$ over $[0,1]$: 
\begin{align}\label{prob:2}
\text{Find } \alpha_n \in [0,1] \text{ such that } 
P_n \left(\alpha_n\right) = \min_{\alpha \in [0,1]} P_n \left(\alpha\right).
\end{align}
When the solution $\alpha_n$ to Problem \eqref{prob:2} can be obtained in each iteration, 
$P_n(\alpha_n) \leq P_n(0)$ holds for all $n \in \mathbb{N}$.
Accordingly, if the next iterate $x_{n+1}$ is defined by $x_{n+1} := x_n (\alpha_n)$,
$\| x_{n+1} - T ( x_{n+1} ) \| \leq \| x_n - T (x_n) \|$ $(n\in\mathbb{N})$ holds, i.e., 
$(\|x_n - T(x_n)  \|)_{n\in\mathbb{N}}$ is monotone decreasing.
Since the exact solution to Problem \eqref{prob:2} cannot be easily obtained,
the step size $\alpha_n$ can be chosen so as to yield an approximate minimum for Problem \eqref{prob:2} in each iteration, specifically, to satisfy the following {\em Wolfe-type} conditions \cite{wolfe1969,wolfe1971}: 
given $x_n, d_n \in H$, and 
$\delta, \sigma \in (0,1)$ with $\delta \leq \sigma$,
\begin{align}
&P_n \left(\alpha_n \right) - P_n \left(0 \right) 
\leq \delta \alpha_n \left\langle Q_n \left(0\right), d_n  \right\rangle, \label{armijo} \\
&\left\langle Q_n \left(\alpha_n \right), d_n \right\rangle 
\geq \sigma \left\langle Q_n \left(0\right), d_n  \right\rangle. \label{wolfe} 
\end{align}

Condition \eqref{armijo} is the Armijo-type condition for $P_n$ (see \eqref{armijo0} for the Armijo-type condition with $d_n := - (x_n - T(x_n))$ for the potential function $g_n$).
Under the conditions that $d_n := - (x_n - T(x_n))$ and $x_{n+1} := x_n(\alpha_n)$ $(n\in\mathbb{N})$, 
Algorithm \eqref{kra} with \eqref{armijo} satisfies
$\| x_{n+1} - T(x_{n+1}) \|^2 \leq (1 - \delta \alpha_n) \|x_n - T(x_n)\|^2$ $(n\in \mathbb{N})$, which implies that, for all $n\in \mathbb{N}$,\footnote{See Theorem \ref{rate}(i) for the details of the convergence rate of the proposed algorithm when $d_n := - (x_n - T(x_n))$ $(n\in\mathbb{N})$.}
\begin{align}\label{rate_steepest}
\left\|x_n - T \left(x_n\right) \right\| = O\left( \left\{  \sum_{k=0}^n \alpha_k \right\}^{-\frac{1}{2}} \right).
\end{align}

Here, let us see how the step size conditions \eqref{Dunn}, \eqref{armijo0}, \eqref{armijo}, and \eqref{wolfe} 
affect the efficiency of Algorithm \eqref{kra}.
Algorithm \eqref{kra} with \eqref{Dunn} satisfies $\| x_{n+1} - T ( x_{n+1} ) \|^2 \leq \| x_n - T (x_n) \|^2$ 
$(n\in\mathbb{N})$ \cite[(5.14)]{b-c}, 
while Algorithm \eqref{kra} with each of \eqref{armijo0} and \eqref{armijo} satisfies 
$\| x_{n+1} - T ( x_{n+1} ) \|^2  < \| x_n - T (x_n) \|^2$ $(n\in\mathbb{N})$. 
Hence, it can be expected that Algorithm \eqref{kra} with each of \eqref{armijo0} and \eqref{armijo} performs better than Algorithm \eqref{kra} with \eqref{Dunn}.
Since the Armijo-type conditions \eqref{armijo0} and \eqref{armijo} are satisfied for all sufficiently small values of $\alpha_n$ \cite[Subchapter 3.1]{noce}, there is a possibility that Algorithm \eqref{kra} with only the Armijo-type condition \eqref{armijo0} does not make reasonable progress.
Meanwhile, \eqref{wolfe} based on the {\em curvature condition} \cite[Subchapter 3.1]{noce} is used to ensure that $\alpha_n$ is not too small and that unacceptably short steps are ruled out.
Therefore, the Wolfe-type conditions \eqref{armijo} and \eqref{wolfe} should be used to secure efficiency of the algorithm. Moreover, even when $\alpha_n$ satisfying \eqref{armijo0} is not small enough, it can be expected that Algorithm \eqref{kra} with the Wolfe-type conditions \eqref{armijo} and \eqref{wolfe} will have a better convergence rate than Algorithm \eqref{kra} with the Armijo-type condition \eqref{armijo0} because of \eqref{rate_mag}, \eqref{rate_steepest}, and $(\alpha - 1/2)^2 \leq \alpha$ $(\alpha \in [(2-\sqrt{3})/2,1])$. Section \ref{sec:3} introduces the line search algorithm \cite[Algorithm 4.6]{lewis2013} to compute step sizes satisfying \eqref{armijo} and \eqref{wolfe} with appropriately chosen $\delta$ and $\sigma$ and gives performance comparisons of Algorithm \eqref{kra} with each of \eqref{Dunn} and \eqref{armijo0} with the one with \eqref{armijo} and \eqref{wolfe}.

The main concern regarding this line search is how the direction $d_n$ should be updated to accelerate the search for a fixed point of $T$. To address this concern, the following problem will be discussed:
\begin{align}\label{application0}
\text{Minimize } f \left(x \right) \text{ subject to } x \in H,
\end{align}
where $f\colon H \to \mathbb{R}$ is convex and Fr\'echet differentiable and $\nabla f \colon H \to H$ is Lipschitz continuous with a constant $L$. 
Let us define $T^{(f)} \colon H \to H$ by 
\begin{align}\label{example}
T^{\left(f\right)} := \mathrm{Id} - \lambda \nabla f,
\end{align}  
where $\mathrm{Id}$ stands for the identity mapping on $H$ and $\lambda > 0$.
The mapping $T^{(f)}$ satisfies the nonexpansivity condition for $\lambda \in (0,2/L]$ \cite[Proposition 2.3]{iiduka_JOTA} and 
$\mathrm{Fix}(T^{(f)})$ coincides with the solution set of Problem \eqref{application0}. 
From $T^{(f)} (x)-x = (x - \lambda \nabla f(x)) -x = - \lambda \nabla f(x)$ 
$(\lambda > 0, x\in H)$,
Algorithm \eqref{kra} for solving Problem \eqref{application0} is 
\begin{align}\label{steepest}
x_{n+1} = x_n + \alpha_n \left( T^{\left(f\right)} \left( x_n \right) - x_n  \right) = x_n - \lambda \alpha_n \nabla f \left(x_n \right).
\end{align}
This means that the direction $d_n^{(f)} := -(x_n - T^{(f)} (x_n)) = - \lambda \nabla f (x_n)$ is the {\em steepest descent direction} of $f$ at $x_n$ and
Algorithm \eqref{kra} with $T^{(f)}$ (i.e., Algorithm \eqref{steepest}) is the {\em steepest descent method} \cite[Subchapter 3.3]{noce} for Problem \eqref{application0}.

There are many algorithms with useful search directions \cite[Chapters 5-19]{noce} to accelerate the steepest descent method for unconstrained optimizations. In particular, algorithms with {\em nonlinear conjugate gradient directions} \cite{hager2006}, \cite[Subchapter 5.2]{noce},
\begin{align}\label{conjugate}
d_{n+1}^{\left(f\right)} := - \nabla f \left(x_{n+1} \right) + \beta_n d_n^{\left(f\right)},
\end{align} 
where $\beta_n \in \mathbb{R}$, have been widely used as efficient accelerated versions for most gradient methods. Well-known formulas for $\beta_n$ include the Hestenes--Stiefel (HS) \cite{HS1952}, Fletcher--Reeves (FR) \cite{FR1964}, Polak--Ribi\`ere--Polyak (PRP) \cite{PR1969,P1969}, and Dai--Yuan (DY) \cite{DY1999} formulas:
\begin{align}\label{beta}
\begin{split}
\beta_n^{\mathrm{HS}} := \frac{\left\langle \nabla f \left(x_{n+1} \right), y_n \right\rangle}{\left\langle d_n, y_n \right\rangle}, \text{ }
\beta_n^{\mathrm{FR}} := \frac{\left\| \nabla f \left(x_{n+1}\right) \right\|^2}{\left\| \nabla f\left(x_n\right) \right\|^2},\\
\beta_n^{\mathrm{PRP}} := \frac{\left\langle \nabla f \left(x_{n+1}\right), y_n \right\rangle}{\left\| \nabla f\left(x_n\right) \right\|^2},
\text{ }
\beta_n^{\mathrm{DY}} := \frac{\left\| \nabla f \left(x_{n+1}\right) \right\|^2}{\left\langle d_n, y_n \right\rangle},
\end{split}
\end{align}
where $y_n := \nabla f (x_{n+1}) - \nabla f(x_n)$.

Motivated by these observations, we decided to use the following direction to accelerate the search for a fixed point of $T$, which can be obtained by replacing $\nabla f$ in \eqref{conjugate} with $\mathrm{Id} - T$ (see also \eqref{example} for the relationship between $\nabla f$ and $T^{(f)}$): given the current direction $d_n \in H$, the current iterate $x_n \in H$, and a step size $\alpha_n$ satisfying \eqref{armijo} and \eqref{wolfe}, the next direction $d_{n+1}$ is defined by 
\begin{align}\label{new}
d_{n+1} := - \left( x_{n+1} - T \left(x_{n+1} \right) \right) + \beta_n d_n,
\end{align}
where $\beta_n$ is given by one of the formulas in \eqref{beta} when $\nabla f = \mathrm{Id} - T$.

This paper proposes iterative algorithms (Algorithm \ref{algo:1}) that use the direction \eqref{new} and step sizes satisfying the Wolfe-type conditions \eqref{armijo} and \eqref{wolfe} for solving Problem \eqref{prob:1} and describes their convergence analyses
(Theorems \ref{thm:1}--\ref{thm:5}). 
We also provide their convergence rate analyses (Theorem \ref{rate}).

The main contribution of this paper is to enable us to propose nonlinear conjugate gradient algorithms for 
{\em constrained smooth convex optimization} which are examples of the proposed line search fixed point algorithms,
in contrast to the previously reported results for nonlinear conjugate gradient algorithms for unconstrained smooth nonconvex optimization \cite[Subchapter 5.2]{noce}, \cite{DY1999,FR1964,HZ2005,hager2006,HS1952,PR1969,P1969}.
Concretely speaking, our nonlinear conjugate gradient algorithms are obtained in the following steps. 
Given a nonempty, closed convex set $C \subset H$ and a convex function $f \colon H \to \mathbb{R}$
with the Lipschitz continuous gradient, 
let us define
\begin{align*}
T := P_C \left( \mathrm{Id} - \lambda \nabla f \right),
\end{align*} 
where $\lambda \in (0,2/L]$, $L$ is the Lipschitz constant of $\nabla f$, and $P_C$ stands for the metric projection onto $C$.
Then, Proposition 2.3 in \cite{iiduka_JOTA} indicates that the mapping $T$ is nonexpansive and satisfies 
\begin{align*}
\mathrm{Fix}\left(T \right) = \argmin_{x\in C} f \left(x\right).
\end{align*}
From \eqref{new} with $T := P_C ( \mathrm{Id} - \lambda \nabla f )$, the proposed nonlinear conjugate gradient algorithms for finding a point in 
$\mathrm{Fix}(T) = \argmin_{x\in C} f(x)$
can be expressed as follows:
given $x_n, d_n \in H$ and $\alpha_n$ satisfying \eqref{armijo} and \eqref{wolfe}, 
\begin{align*}
&x_{n+1} := x_n \left( \alpha_n \right) = x_n + \alpha_n d_n,\\
&d_{n+1} := - \left( x_{n+1} -  P_C \left( x_{n+1} - \lambda \nabla f \left(x_{n+1} \right) \right) \right) + \beta_n d_n,
\end{align*}
where $\beta_n \in \mathbb{R}$ is each of the following formulas:\footnote{To guarantee the convergence of the PRP and HS methods 
for unconstrained optimization, the formulas $\beta_n^{\mathrm{PRP}+} := \max \{\beta_n^{\mathrm{PRP}}, 0\}$ 
and $\beta_n^{\mathrm{HS}+} := \max \{\beta_n^{\mathrm{HS}}, 0\}$ were presented in \cite{powell1984}.
We use the modifications to perform the convergence analyses on the proposed line search fixed point algorithms.}
\begin{align}\label{formulas}
\begin{split}
&\beta_n^{\mathrm{HS}+} := \max \left\{\frac{\left\langle x_{n+1} - P_C \left( x_{n+1} - \lambda \nabla f \left(x_{n+1} \right) \right), y_n \right\rangle}{\left\langle d_n, y_n \right\rangle}, 0 \right\},\\
&\beta_n^{\mathrm{FR}} := \frac{\left\| x_{n+1} - P_C \left( x_{n+1} - \lambda \nabla f \left(x_{n+1} \right) \right) \right\|^2}{\left\| x_{n} - P_C \left( x_{n} - \lambda \nabla f \left(x_{n} \right) \right) \right\|^2},\\
&\beta_n^{\mathrm{PRP}+} := \max \left\{\frac{\left\langle x_{n+1} - P_C \left( x_{n+1} - \lambda \nabla f \left(x_{n+1} \right) \right), y_n \right\rangle}{\left\| x_{n} - P_C \left( x_{n} - \lambda \nabla f \left(x_{n} \right) \right) \right\|^2}, 0 \right\},\\
&\beta_n^{\mathrm{DY}} := \frac{\left\| x_{n+1} - P_C \left( x_{n+1} - \lambda \nabla f \left(x_{n+1} \right) \right) \right\|^2}{\left\langle d_n, y_n \right\rangle},
\end{split}
\end{align}
where $y_n := (x_{n+1} - P_C ( x_{n+1} - \lambda \nabla f (x_{n+1}) )) - (x_n - P_C ( x_{n} - \lambda \nabla f (x_{n}) ))$.
Our convergence analyses are performed by referring to useful results on unconstrained smooth nonconvex optimization 
(see \cite{Al-Baali1985,DY1999,gilbert1992,hager2006,wolfe1969,wolfe1971,zou1970} and references therein)
because the proposed fixed point algorithms are based on the steepest descent and nonlinear conjugate gradient directions for unconstrained smooth nonconvex optimization (see \eqref{application0}--\eqref{new}).
We would like to emphasize that 
combining unconstrained smooth nonconvex optimization theory with fixed point theory for nonexpansive mappings
enables us to develop the novel nonlinear conjugate gradient algorithms for constrained smooth convex optimization.
The nonlinear conjugate gradient algorithms are a concrete response to the issue of constrained smooth convex optimization that is whether or not we can present nonlinear conjugate gradient algorithms to solve constrained smooth convex optimization problems.

To verify whether the proposed nonlinear conjugate gradient algorithms are accelerations for solving practical problems, we apply them to constrained quadratic programming problems (Subsection \ref{subsec:3.2}) and 
generalized convex feasibility problems (Subsection \ref{subsec:3.1}) (see \cite{com1999,yamada} and references therein for the relationship between the generalized convex feasibility problem and signal processing problems), which are constrained smooth convex optimization problems and particularly interesting applications of Problem \eqref{prob:1}. Moreover, we numerically compare their abilities to solve concrete constrained quadratic programming problems and
generalized convex feasibility problems with those of previous algorithms based on the Krasnosel'ski\u\i-Mann algorithm (Algorithm \eqref{kra} with step sizes satisfying \eqref{Dunn} and Algorithm \eqref{kra} with step sizes satisfying \eqref{armijo0}) and 
show that they can find optimal solutions to these problems faster than the previous ones.

Throughout this paper, we shall let $\mathbb{N}$ be the set of zero and all positive integers, $\mathbb{R}^d$ be a $d$-dimensional Euclidean space, $H$ be a real Hilbert space with inner product $\langle \cdot, \cdot \rangle$ and its induced norm $\| \cdot \|$, and $T\colon H \to H$ be a nonexpansive mapping with $\mathrm{Fix}(T) := \{ x\in H \colon T(x) = x \} \neq \emptyset$.

\section{Line search fixed point algorithms based on nonlinear conjugate gradient directions}\label{sec:2}
Let us begin by explicitly stating our algorithm for solving Problem \eqref{prob:1} discussed in Section \ref{sec:1}.

\begin{algo}\label{algo:1}
\text{}

Step 0.
Take $\delta, \sigma \in (0,1)$ with $\delta \leq \sigma$. 
Choose $x_0 \in H$ arbitrarily and set $d_0 := -(x_0 - T(x_0))$ and $n:= 0$.

Step 1.
Compute $\alpha_n \in (0,1]$ satisfying 
\begin{align}
\left\| x_n \left( \alpha_n  \right) - T \left( x_n \left( \alpha_n  \right) \right)  \right\|^2 
- \left\| x_n - T \left( x_n \right)  \right\|^2
&\leq \delta \alpha_n \left\langle  x_n - T\left(x_n\right), d_n \right\rangle,\label{Armijo} \\
\left\langle  x_n\left( \alpha_n  \right) - T\left(x_n \left( \alpha_n  \right) \right), d_n \right\rangle
&\geq \sigma \left\langle  x_n - T\left(x_n \right), d_n \right\rangle,\label{Wolfe}
\end{align}
where $x_n(\alpha_n) := x_n + \alpha_n d_n$.
Compute $x_{n+1} \in H$ by
\begin{align}\label{xn}
x_{n+1} := x_n + \alpha_n d_n.
\end{align}

Step 2. 
If $\| x_{n+1} - T(x_{n+1}) \|= 0$, stop.
Otherwise, go to Step 3.

Step 3. 
Compute $\beta_n \in \mathbb{R}$ by using each of the following formulas:
\begin{align}\label{Beta}
\begin{split}
&\beta_n^{\mathrm{SD}} := 0,\\
&\beta_n^{\mathrm{HS}+} := \max \left\{\frac{\left\langle x_{n+1} - T \left(x_{n+1}\right), y_n \right\rangle}{\left\langle d_n, y_n \right\rangle}, 0 \right\}, \text{ }
\beta_n^{\mathrm{FR}} := \frac{\left\| x_{n+1} - T \left(x_{n+1}\right) \right\|^2}{\left\| x_{n} - T \left(x_{n}\right) \right\|^2},\\
&\beta_n^{\mathrm{PRP}+} := \max \left \{\frac{\left\langle x_{n+1} - T \left(x_{n+1}\right), y_n \right\rangle}{\left\| x_{n} - T \left(x_{n}\right) \right\|^2}, 0 \right\},
\text{ }
\beta_n^{\mathrm{DY}} := \frac{\left\| x_{n+1} - T \left(x_{n+1}\right) \right\|^2}{\left\langle d_n, y_n \right\rangle},
\end{split}
\end{align}
where $y_n := (x_{n+1} - T(x_{n+1})) - (x_n - T(x_n))$.
Generate $d_{n+1} \in H$ by
\begin{align*}
d_{n+1} := - \left( x_{n+1} - T \left(x_{n+1}\right) \right) + \beta_n d_n.
\end{align*}

Step 4. 
Put $n := n+1$ and go to Step 1.
\end{algo}

We need to use appropriate line search algorithms to compute $\alpha_n$ $(n\in \mathbb{N})$ satisfying \eqref{Armijo} and \eqref{Wolfe}. In Section \ref{sec:3}, we use a useful one (Algorithm \ref{line_search}) \cite[Algorithm 4.6]{lewis2013} that can obtain the step sizes satisfying \eqref{Armijo} and \eqref{Wolfe} whenever the line search algorithm terminates \cite[Theorem 4.7]{lewis2013}. Although the efficiency of the line search algorithm depends on the parameters $\delta$ and $\sigma$, thanks to the reference \cite[Subsection 6.1]{lewis2013}, we can set appropriate $\delta$ and $\sigma$ before executing it \cite[Algorithm 4.6]{lewis2013} and Algorithm \ref{algo:1}. See Section \ref{sec:3} for the numerical performance of the line search algorithm \cite[Algorithm 4.6]{lewis2013} and Algorithm \ref{algo:1}.

It can be seen that Algorithm \ref{algo:1} is well-defined when $\beta_n$ is defined by $\beta_n^{\mathrm{SD}}$, $\beta_n^{\mathrm{FR}}$, or $\beta_n^{\mathrm{PRP}+}$.
The discussion in Subsection \ref{subsec:DY} shows that Algorithm \ref{algo:1} with $\beta_n = \beta_n^{\mathrm{DY}}$ is well-defined (Lemma \ref{lem:DY}(i)). Moreover, it is guaranteed that under certain assumptions, Algorithm \ref{algo:1} with $\beta_n = \beta_n^{\mathrm{HS}+}$ is well-defined (Theorem \ref{thm:5}).

\subsection{Algorithm \ref{algo:1} with $\beta_n = \beta_n^{\mathrm{SD}}$}\label{subsec:SD}
This subsection considers Algorithm \ref{algo:1} with $\beta_n^{\mathrm{SD}}$ $(n\in\mathbb{N})$, 
which is based on
the steepest descent (SD) direction (see \eqref{steepest}), i.e.,
\begin{align}\label{KM}
x_{n+1} := x_n + \alpha_n \left( T \left(x_n\right) - x_n \right) \text{ } \left( n\in \mathbb{N} \right).
\end{align}
Theorems 4 and 8 in \cite{mag2004} indicate that, if $(\alpha_n)_{n\in\mathbb{N}}$ satisfies  
the Armijo-type condition \eqref{armijo0},
Algorithm \eqref{KM} converges to a fixed point of $T$.
The following theorem says that Algorithm \eqref{KM}, with $(\alpha_n)_{n\in\mathbb{N}}$ satisfying the Wolfe-type conditions \eqref{Armijo} and \eqref{Wolfe},
converges to a fixed point of $T$.

\begin{thm}\label{thm:1}
Suppose that $(x_n)_{n\in\mathbb{N}}$ is the sequence generated by Algorithm \ref{algo:1} with $\beta_n = \beta_n^{\mathrm{SD}}$ $(n\in\mathbb{N})$. Then, $(x_n)_{n\in\mathbb{N}}$ either terminates at a fixed point of $T$ or 
\begin{align*}
\lim_{n\to \infty} \left\|x_n - T \left(x_n \right) \right\| = 0.
\end{align*}
In the latter situation,
$(x_n)_{n\in\mathbb{N}}$ weakly converges to a fixed point of $T$.
\end{thm}

\subsubsection{Proof of Theorem \ref{thm:1}}
If $m \in \mathbb{N}$ exists such that $\| x_{m} - T(x_{m}) \| = 0$, Theorem \ref{thm:1} holds.
Accordingly, it can be assumed that, for all $n\in\mathbb{N}$, $\| x_n - T (x_n) \| \neq 0$ holds.

First, the following lemma can be proven by referring to \cite{wolfe1969,wolfe1971,zou1970}.
\begin{lem}\label{Zou}
Let $(x_n)_{n\in\mathbb{N}}$ and $(d_n)_{n\in\mathbb{N}}$ be the sequences generated by Algorithm \ref{algo:1}.
Assume that $\langle x_n - T(x_n), d_n \rangle < 0$ for all $n\in \mathbb{N}$.
Then, 
\begin{align*}
\sum_{n=0}^{\infty} \left( \frac{\left\langle x_n - T(x_n), d_n \right\rangle}{\left\| d_n \right\|} \right)^2 < \infty.
\end{align*}
\end{lem}

\begin{proof}
The Cauchy-Schwarz inequality and the triangle inequality ensure that, for all $n\in\mathbb{N}$,
$\langle d_n, ( x_{n+1} - T ( x_{n+1}) ) - (x_n - T (x_n ) ) \rangle
\leq 
\| d_n \| \| ( x_{n+1} - T ( x_{n+1}) ) -  (x_n - T (x_n ) ) \|
\leq 
\| d_n \| ( \| T ( x_{n} ) - T (x_{n+1} ) \| + \| x_{n+1} - x_n \| )$,
which, together with the nonexpansivity of $T$ and \eqref{xn}, implies that, for all $n\in \mathbb{N}$,
\begin{align*}
\left\langle d_n, \left( x_{n+1} - T \left( x_{n+1}  \right) \right) - \left(x_n - T \left(x_n\right) \right)  \right\rangle
\leq 2 \alpha_n \left\| d_n \right\|^2.
\end{align*}
Moreover, \eqref{Wolfe} means that, for all $n\in\mathbb{N}$,
\begin{align*}
\left\langle d_n, \left( x_{n+1} - T \left( x_{n+1}  \right) \right) - \left(x_n - T \left(x_n\right) \right)  \right\rangle
\geq \left(\sigma -1 \right)\left\langle d_n, x_{n} - T \left( x_{n}  \right) \right\rangle.
\end{align*}
Accordingly, for all $n\in\mathbb{N}$,
\begin{align*}
\left(\sigma -1 \right)\left\langle d_n, x_{n} - T \left( x_{n}  \right) \right\rangle \leq 2 \alpha_n \left\| d_n \right\|^2.
\end{align*}
Since $\|d_n\| \neq 0$ $(n\in\mathbb{N})$ holds from $\langle x_n - T(x_n), d_n \rangle < 0$ $(n\in\mathbb{N})$, we find that, for all $n\in\mathbb{N}$, 
\begin{align}\label{eq:1}
\frac{\left(\sigma -1 \right)\left\langle d_n, x_{n} - T \left( x_{n}  \right) \right\rangle}{2 \left\| d_n \right\|^2} \leq \alpha_n.
\end{align}
Condition \eqref{Armijo} means that, for all $n\in\mathbb{N}$, 
$\| x_{n+1} - T(x_{n+1} )\|^2 - \|x_n - T (x_n) \|^2 
\leq \delta \alpha_n \langle x_n - T (x_n ), d_n \rangle$,
which, together with $\langle x_n - T(x_n), d_n \rangle < 0$ $(n\in\mathbb{N})$, implies that, for all $n\in\mathbb{N}$,
\begin{align}\label{eq:2}
\alpha_n \leq \frac{\left\|x_n - T\left(x_n\right) \right\|^2 - \left\| x_{n+1} - T \left(x_{n+1} \right) \right\|^2}{-\delta \left\langle x_n - T \left(x_n \right), d_n \right\rangle}.
\end{align}
From \eqref{eq:1} and \eqref{eq:2}, for all $n\in\mathbb{N}$,
\begin{align*}
\frac{\left(\sigma -1 \right)\left\langle d_n, x_{n} - T \left( x_{n}  \right) \right\rangle}{2 \left\| d_n \right\|^2}
\leq 
\frac{\left\|x_n - T\left(x_n\right) \right\|^2 - \left\| x_{n+1} - T \left(x_{n+1} \right) \right\|^2}{-\delta \left\langle x_n - T \left(x_n \right), d_n \right\rangle},
\end{align*}
which implies that, for all $n\in\mathbb{N}$,
\begin{align*}
\frac{\delta \left(1 - \sigma \right)\left\langle d_n, x_{n} - T \left( x_{n}  \right) \right\rangle^2}{2 \left\| d_n \right\|^2}
\leq 
\left\|x_n - T\left(x_n\right) \right\|^2 - \left\| x_{n+1} - T \left(x_{n+1} \right) \right\|^2.
\end{align*}
Summing up this inequality from $n=0$ to $n=N \in \mathbb{N}$ guarantees that, for all $N\in \mathbb{N}$,
\begin{align*}
\frac{\delta \left(1 - \sigma \right)}{2} \sum_{n=0}^N \frac{\left\langle d_n, x_{n} - T \left( x_{n}  \right) \right\rangle^2}{\left\| d_n \right\|^2}
&\leq \left\|x_0 - T\left(x_0\right) \right\|^2 - \left\| x_{N+1} - T \left(x_{N+1} \right) \right\|^2\\
&\leq \left\|x_0 - T\left(x_0\right) \right\|^2 < \infty.
\end{align*}
Therefore, the conclusion in Lemma \ref{Zou} is satisfied.
\end{proof}

Lemma \ref{Zou} leads to the following.

\begin{lem}\label{lem:sd}
Suppose that the assumptions of Theorem \ref{thm:1} are satisfied. Then,
\begin{enumerate}
\item[{\em (i)}]
$\lim_{n\to \infty} \| x_n - T(x_n) \|= 0$.
\item[{\em (ii)}]
$(\| x_n - x \|)_{n\in\mathbb{N}}$ is monotone decreasing for all $x\in \mathrm{Fix}(T)$.
\item[{\em (iii)}]
$(x_n)_{n\in\mathbb{N}}$ weakly converges to a point in $\mathrm{Fix}(T)$.
\end{enumerate} 
\end{lem}

Items (i) and (iii) in Lemma \ref{lem:sd} indicate that Theorem \ref{thm:1} holds under the assumption that  
$\| x_n - T (x_n) \| \neq 0$ $(n\in\mathbb{N})$.

\begin{proof}
(i) 
In the case where $\beta_n := \beta_n^{\mathrm{SD}} = 0$ $(n\in\mathbb{N})$, $d_n = - (x_n - T(x_n))$ holds for all $n\in\mathbb{N}$.
Hence, $\langle x_n - T(x_n), d_n \rangle = - \|x_n - T(x_n)\|^2 < 0$ $(n\in\mathbb{N})$.
Lemma \ref{Zou} thus guarantees that
$\sum_{n=0}^\infty \| x_{n} - T ( x_{n} ) \|^2 < \infty$,
which implies $\lim_{n\to\infty} \| x_n - T(x_n) \|= 0$.

(ii)
The triangle inequality and the nonexpansivity of $T$ ensure that, for all $n\in \mathbb{N}$ and for all $x\in \mathrm{Fix}(T)$,
$\| x_{n+1} - x \|
= \| x_n + \alpha_n ( T (x_n) - x_n ) - x \|
\leq (1-\alpha_n ) \| x_n - x \| + \alpha_n \|T (x_n) - T (x)\|
\leq \| x_n - x \|$.

(iii) 
Lemma \ref{lem:sd}(ii) means that $\lim_{n\to\infty} \|x_n - x\|$ exists for all $x\in \mathrm{Fix}(T)$. Accordingly, $(x_n)_{n\in\mathbb{N}}$ is bounded. Hence, there is a subsequence $(x_{n_k})_{k\in\mathbb{N}}$ of $(x_n)_{n\in\mathbb{N}}$ such that $(x_{n_k})_{k\in\mathbb{N}}$ weakly converges to a point $x^* \in H$. Here, let us assume that $x^* \notin \mathrm{Fix}(T)$. Then, Opial's condition \cite[Lemma 1]{opial}, Lemma \ref{lem:sd}(i), and the nonexpansivity of $T$ guarantee that
\begin{align*}
\liminf_{k\to\infty} \left\| x_{n_{k}} - x^* \right\|
&< \liminf_{k\to\infty} \left\| x_{n_{k}} - T \left(x^* \right) \right\|\\
&= \liminf_{k\to\infty} \left\| x_{n_{k}} - T \left( x_{n_{k}} \right)
     + T \left( x_{n_{k}} \right) - T \left(x^* \right) \right\|\\
&=  \liminf_{k\to\infty} \left\|   T \left( x_{n_{k}} \right) - T \left(x^* \right) \right\|\\
&\leq  \liminf_{k\to\infty} \left\|    x_{n_{k}} - x^*  \right\|,
\end{align*}
which is a contradiction.
Hence, $x^* \in \mathrm{Fix}(T)$.
Let us take another subsequence $(x_{n_{i}})_{i\in \mathbb{N}}$ $(\subset (x_{n})_{n\in \mathbb{N}})$
which weakly converges to $x_{*} \in H$.
A similar discussion to the one for obtaining $x^* \in \mathrm{Fix}(T)$ ensures that $x_{*} \in \mathrm{Fix}(T)$.
Assume that $x^* \neq x_{*}$.
The existence of $\lim_{n\to\infty} \| x_n - x \|$ $(x\in \mathrm{Fix}(T))$ and Opial's condition \cite[Lemma 1]{opial} imply that
\begin{align*}
\lim_{n\to\infty} \left\| x_n - x^*  \right\| 
&= \lim_{k\to\infty} \left\| x_{n_{k}} - x^*  \right\|
< \lim_{k\to\infty} \left\| x_{n_{k}} - x_{*}  \right\|\\
&= \lim_{n\to\infty} \left\| x_{n} - x_{*}  \right\|
= \lim_{i\to\infty} \left\| x_{n_{i}} - x_{*}  \right\|\\
&< \lim_{i\to\infty} \left\| x_{n_{i}} - x^*  \right\|
= \lim_{n\to\infty} \left\| x_n - x^*  \right\|,
\end{align*}
which is a contradiction. 
Therefore, $x^* = x_*$.
Since any subsequence of $(x_n)_{n\in\mathbb{N}}$ weakly converges to the same fixed point of $T$, 
it is guaranteed that the whole $(x_n)_{n\in\mathbb{N}}$ weakly converges to a fixed point of $T$.
This completes the proof.
\end{proof}

\subsection{Algorithm \ref{algo:1} with $\beta_n = \beta_n^{\mathrm{DY}}$}\label{subsec:DY}
The following is a convergence analysis of Algorithm \ref{algo:1} with $\beta_n = \beta_n^{\mathrm{DY}}$.

\begin{thm}\label{thm:2}
Suppose that $(x_n)_{n\in\mathbb{N}}$ is the sequence generated by Algorithm \ref{algo:1} with $\beta_n = \beta_n^{\mathrm{DY}}$ $(n\in\mathbb{N})$. 
Then, $(x_n)_{n\in\mathbb{N}}$ either terminates at a fixed point of $T$ or 
\begin{align*}
\lim_{n\to \infty} \left\|x_n - T \left(x_n \right) \right\| = 0.
\end{align*}
\end{thm}

\subsubsection{Proof of Theorem \ref{thm:2}}
Since the existence of $m\in\mathbb{N}$ such that $\| x_m - T(x_m) \| = 0$ implies that Theorem \ref{thm:2} holds, it can be assumed that, for all $n\in\mathbb{N}$, $\| x_n - T (x_n) \| \neq 0$ holds. Theorem \ref{thm:2} can be proven by using the ideas presented in the proof of \cite[Theorem 3.3]{DY1999}. The proof of Theorem \ref{thm:2} is divided into three steps.

\begin{lem}\label{lem:DY}
Suppose that the assumptions in Theorem \ref{thm:2} are satisfied.
Then,
\begin{enumerate}
\item[{\em (i)}]
$\langle x_n - T(x_n), d_n \rangle < 0$ $(n\in\mathbb{N})$.
\item[{\em (ii)}]
$\liminf_{n\to \infty} \| x_n - T(x_n) \|= 0$.
\item[{\em (iii)}]
$\lim_{n\to \infty} \| x_n - T(x_n) \|= 0$.
\end{enumerate} 
\end{lem}

\begin{proof}
(i) 
From $d_0 := - (x_0 - T(x_0))$, 
$\langle x_0 - T(x_0), d_0 \rangle = - \|x_0 - T(x_0)  \|^2 < 0$.
Suppose that $\langle x_n - T(x_n), d_n \rangle < 0$ holds for some $n\in\mathbb{N}$.
Accordingly, the definition of $y_n:= (x_{n+1} - T(x_{n+1})) - (x_n - T(x_n))$ and \eqref{Wolfe} ensure that
\begin{align*}
\left\langle d_n, y_n \right\rangle
&=  \left\langle d_n, x_{n+1} - T \left(x_{n+1}\right) \right\rangle - \left\langle d_n, x_{n} - T\left(x_{n}\right) \right\rangle\\
&\geq \left( \sigma - 1  \right) \left\langle d_n, x_{n} - T\left(x_{n}\right) \right\rangle > 0,
\end{align*}
which implies that 
\begin{align*}
\beta_n^{\mathrm{DY}} := \frac{\left\| x_{n+1} - T \left(x_{n+1}\right) \right\|^2}{\left\langle d_n, y_n \right\rangle} > 0.
\end{align*}
From the definition of $d_{n+1} := - (x_{n+1} - T(x_{n+1})) + \beta_n^{\mathrm{DY}} d_n$,
we have
\begin{align*}
\left\langle d_{n+1}, x_{n+1} - T \left(x_{n+1}\right) \right\rangle
&= - \left\| x_{n+1} - T \left(x_{n+1}\right) \right\|^2 + \beta_n^{\mathrm{DY}} 
\left\langle d_n, x_{n+1} - T \left(x_{n+1}\right) \right\rangle\\
&= \left\| x_{n+1} - T \left(x_{n+1}\right) \right\|^2 
 \left\{ -1 + \frac{\left\langle d_n, x_{n+1} - T \left(x_{n+1}\right) \right\rangle}{\left\langle d_n, y_n \right\rangle} \right\}\\
&= \left\| x_{n+1} - T \left(x_{n+1}\right) \right\|^2
 \frac{\left\langle d_n, \left( x_{n+1} - T \left(x_{n+1}\right) \right) - y_n \right\rangle}{\left\langle d_n, y_n \right\rangle},
\end{align*}
which, together with the definitions of $y_n$ and $\beta_n^{\mathrm{DY}}$ $(>0)$, implies that
\begin{align}\label{DY1}
\begin{split} 
\left\langle d_{n+1}, x_{n+1} - T \left(x_{n+1}\right) \right\rangle 
&= \left\| x_{n+1} - T \left(x_{n+1}\right) \right\|^2
  \frac{\left\langle d_n, x_{n} - T \left(x_{n}\right) \right\rangle}{\left\langle d_n, y_n \right\rangle}\\
&=  \beta_n^{\mathrm{DY}} \left\langle d_n, x_{n} - T \left(x_{n}\right) \right\rangle < 0.
\end{split}
\end{align}
Induction shows that $\langle x_n - T(x_n), d_n \rangle < 0$ for all $n\in\mathbb{N}$.
This implies $\beta_n^{\mathrm{DY}} > 0$ $(n\in\mathbb{N})$; i.e., 
Algorithm \ref{algo:1} with $\beta_n = \beta_n^{\mathrm{DY}}$ is well-defined.

(ii)
Assume that $\liminf_{n\to\infty} \| x_n - T(x_n) \| > 0$.
Then, there exist $n_0 \in \mathbb{N}$ and $\varepsilon > 0$ such that 
$\|x_n - T(x_n)\| \geq \varepsilon$ for all $n \geq n_0$.
Since we have assumed that $\|x_n - T(x_n)\| \neq 0$ $(n\in\mathbb{N})$, we may further assume that 
$\|x_n - T(x_n)\| \geq \varepsilon$ for all $n \in \mathbb{N}$.
From the definition of $d_{n+1} := - (x_{n+1} - T(x_{n+1})) + \beta_n^{\mathrm{DY}} d_n$ $(n\in\mathbb{N})$,
we have, for all $n\in\mathbb{N}$, 
\begin{align*}
\beta_n^{\mathrm{DY}^2} \left\| d_n \right\|^2
&= \left\| d_{n+1} + \left(x_{n+1} - T \left(x_{n+1} \right) \right) \right\|^2\\
&= \left\| d_{n+1} \right\|^2 + 2 \left\langle d_{n+1}, x_{n+1} - T \left(x_{n+1} \right) \right\rangle 
   + \left\| x_{n+1} - T \left(x_{n+1} \right) \right\|^2.
\end{align*}
Lemma \ref{lem:DY}(i) and \eqref{DY1} mean that, for all $n\in\mathbb{N}$,
\begin{align*}
\beta_n^{\mathrm{DY}} = \frac{\left\langle d_{n+1}, x_{n+1} - T \left(x_{n+1}\right) \right\rangle}{\left\langle d_n, x_{n} - T \left(x_{n}\right) \right\rangle}.
\end{align*}
Hence, for all $n\in\mathbb{N}$,
\begin{align*}
&\quad \frac{\left\| d_{n+1} \right\|^2}{\left\langle d_{n+1}, x_{n+1} - T \left(x_{n+1} \right) \right\rangle^2}\\
&= - \frac{\left\| x_{n+1} - T \left(x_{n+1} \right) \right\|^2}{\left\langle d_{n+1}, x_{n+1} - T \left(x_{n+1} \right) \right\rangle^2}
 - \frac{2}{\left\langle d_{n+1}, x_{n+1} - T \left(x_{n+1} \right) \right\rangle} 
  + \frac{\left\| d_n \right\|^2}{\left\langle d_{n}, x_{n} - T \left(x_{n} \right) \right\rangle^2}\\
&= \frac{\left\| d_n \right\|^2}{\left\langle d_{n}, x_{n} - T \left(x_{n} \right) \right\rangle^2}
  + \frac{1}{\left\| x_{n+1} - T \left(x_{n+1} \right) \right\|^2}\\
&\quad - \left\{ \frac{1}{\left\| x_{n+1} - T \left(x_{n+1} \right) \right\|} +   
   \frac{\left\| x_{n+1} - T \left(x_{n+1} \right) \right\|}{\left\langle d_{n+1}, x_{n+1} - T \left(x_{n+1} \right) \right\rangle} \right\}^2\\
&\leq \frac{\left\| d_n \right\|^2}{\left\langle d_{n}, x_{n} - T \left(x_{n} \right) \right\rangle^2}
  + \frac{1}{\left\| x_{n+1} - T \left(x_{n+1} \right) \right\|^2}.     
\end{align*}
Summing up this inequality from $n=0$ to $n=N\in \mathbb{N}$ yields, for all $N\in \mathbb{N}$,
\begin{align*}
\frac{\left\| d_{N+1} \right\|^2}{\left\langle d_{N+1}, x_{N+1} - T \left(x_{N+1} \right) \right\rangle^2}
\leq  \frac{\left\| d_0 \right\|^2}{\left\langle d_{0}, x_{0} - T \left(x_{0} \right) \right\rangle^2}
  + \sum_{k=1}^{N+1} \frac{1}{\left\| x_{k} - T \left(x_{k} \right) \right\|^2},
\end{align*}
which, which together with $\|x_n - T(x_n)\| \geq \varepsilon$ ($n \in \mathbb{N}$) and $d_0 := -(x_0- T(x_0))$,  
implies that, for all $N\in\mathbb{N}$,
\begin{align*}
\frac{\left\| d_{N+1} \right\|^2}{\left\langle d_{N+1}, x_{N+1} - T \left(x_{N+1} \right) \right\rangle^2}
\leq \sum_{k=0}^{N+1} \frac{1}{\left\| x_k- T \left(x_k \right) \right\|^2}
\leq \frac{N+2}{\varepsilon^2}.
\end{align*}
Since Lemma \ref{lem:DY}(i) implies $\| d_n \| \neq 0$ $(n\in\mathbb{N})$, 
we have, for all $N \in \mathbb{N}$,
\begin{align*}
\frac{\left\langle d_{N+1}, x_{N+1} - T \left(x_{N+1} \right) \right\rangle^2}{\left\| d_{N+1} \right\|^2}
\geq \frac{\varepsilon^2}{N+2}.
\end{align*}
Therefore, Lemma \ref{Zou} guarantees that 
\begin{align*}
\infty > \sum_{k=1}^\infty \left( \frac{\left\langle d_{k}, x_{k} - T \left(x_{k} \right) \right\rangle}{\left\| d_{k} \right\|} \right)^2 
\geq \sum_{k=1}^\infty \frac{\varepsilon^2}{k+1} = \infty.
\end{align*}
This is a contradiction.
Hence, $\liminf_{n\to \infty} \|x_n - T(x_n)\| =0$.

(iii)
Condition \eqref{Armijo} and Lemma \ref{lem:DY}(i) lead to that, for all $n\in\mathbb{N}$,
\begin{align*}
\left\| x_{n+1} - T\left( x_{n+1}  \right) \right\|^2 - \left\| x_{n} - T\left( x_{n}  \right) \right\|^2
\leq \delta \alpha_n \left\langle x_n - T \left(x_n \right), d_n \right\rangle < 0. 
\end{align*}
Accordingly, $(\| x_n - T(x_n) \|)_{n\in\mathbb{N}}$ is monotone decreasing; i.e., 
there exists $\lim_{n\to \infty} \| x_n - T(x_n) \|$.
Lemma \ref{lem:DY}(ii) thus ensures that $\lim_{n\to\infty} \|x_n - T(x_n)\| = 0$.
This completes the proof.
\end{proof}

\subsection{Algorithm \ref{algo:1} with $\beta_n = \beta_n^{\mathrm{FR}}$}\label{subsec:FR}
To establish the convergence of Algorithm \ref{algo:1} when $\beta_n = \beta_n^{\mathrm{FR}}$, 
we assume that the step sizes $\alpha_n$ satisfy the {\em strong Wolfe-type conditions}, which are \eqref{Armijo} and 
the following strengthened version of \eqref{Wolfe}: for $\sigma \leq 1/2$,
\begin{align}\label{s_Wolfe}
\left| \left\langle x_n \left(\alpha_n \right) - T \left(x_n \left(\alpha_n \right)\right), d_n \right\rangle \right|
\leq - \sigma  \left\langle x_n - T \left(x_n \right), d_n \right\rangle.
\end{align}
See \cite{Al-Baali1985} on the global convergence of the FR method for unconstrained optimization
under the strong Wolfe conditions.

The following is a convergence analysis of Algorithm \ref{algo:1} with $\beta_n = \beta_n^{\mathrm{FR}}$.

\begin{thm}\label{thm:3}
Suppose that
$(x_n)_{n\in\mathbb{N}}$ is the sequence generated by Algorithm \ref{algo:1} with $\beta_n = \beta_n^{\mathrm{FR}}$ $(n\in\mathbb{N})$, where $(\alpha_n)_{n\in\mathbb{N}}$ satisfies \eqref{Armijo} and \eqref{s_Wolfe}. 
Then $(x_n)_{n\in\mathbb{N}}$ either terminates at a fixed point of $T$ or 
\begin{align*}
\lim_{n\to \infty} \left\|x_n - T \left(x_n \right) \right\| = 0.
\end{align*}
\end{thm}

\subsubsection{Proof of Theorem \ref{thm:3}}
It can be assumed that, for all $n\in\mathbb{N}$, $\| x_n - T (x_n) \| \neq 0$ holds.
Theorem \ref{thm:3} can be proven by using the ideas in the proof of \cite[Theorem 2]{Al-Baali1985}.

\begin{lem}\label{lem:FR}
Suppose that the assumptions in Theorem \ref{thm:3} are satisfied.
Then,
\begin{enumerate}
\item[{\em (i)}]
$\langle x_n - T(x_n), d_n \rangle < 0$ $(n\in\mathbb{N})$.
\item[{\em (ii)}]
$\liminf_{n\to \infty} \| x_n - T(x_n) \|= 0$.
\item[{\em (iii)}]
$\lim_{n\to \infty} \| x_n - T(x_n) \|= 0$.
\end{enumerate} 
\end{lem}

\begin{proof}
(i) 
Let us show that, for all $n\in\mathbb{N}$,
\begin{align}\label{inequality:1}
- \sum_{j=0}^n \sigma^j \leq \frac{\left\langle x_n - T\left(x_n \right), d_n \right\rangle}{\left\| x_n - T\left(x_n \right)
\right\|^2} \leq -2 + \sum_{j=0}^n \sigma^j. 
\end{align}
From $d_0 := - (x_0 - T(x_0))$, \eqref{inequality:1} holds for $n:= 0$ and $\langle x_0 - T(x_0), d_0 \rangle < 0$.
Suppose that \eqref{inequality:1} holds for some $n\in \mathbb{N}$.
Accordingly, from $\sum_{j=0}^n \sigma^j < \sum_{j=0}^\infty \sigma^j = 1/(1-\sigma)$ and $\sigma \in (0,1/2]$,
we have
\begin{align*} 
\frac{\left\langle x_n - T\left(x_n \right), d_n \right\rangle}{\left\| x_n - T\left(x_n \right)
\right\|^2} < -2 + \sum_{j=0}^\infty \sigma^j = \frac{- \left( 1 - 2 \sigma \right)}{1-\sigma} \leq 0,
\end{align*}
which implies that $\langle x_n - T (x_n ), d_n \rangle < 0$.
The definitions of $d_{n+1}$ and $\beta_n^{\mathrm{FR}}$ enable us to deduce that 
\begin{align*}
\frac{\left\langle x_{n+1} - T\left(x_{n+1} \right), d_{n+1} \right\rangle}{\left\| x_{n+1} - T \left( x_{n+1} \right)  \right\|^2}
&= \frac{\left\langle x_{n+1} - T\left(x_{n+1} \right), - \left( x_{n+1} - T\left(x_{n+1} \right) \right) 
    + \beta_n^{\mathrm{FR}} d_n \right\rangle}{\left\| x_{n+1} - T \left( x_{n+1} \right)  \right\|^2}\\
&= -1 + \frac{\left\| x_{n+1} - T \left( x_{n+1} \right)  \right\|^2}{\left\| x_{n} - T \left( x_{n} \right)  \right\|^2}  
   \frac{\left\langle x_{n+1} - T\left(x_{n+1} \right), d_n \right\rangle}{\left\| x_{n+1} - T \left( x_{n+1} \right)  \right\|^2}\\
&= -1 + \frac{\left\langle x_{n+1} - T\left(x_{n+1} \right), d_n \right\rangle}{\left\| x_{n} - T \left( x_{n} \right)  \right\|^2}.   
\end{align*}
Since \eqref{s_Wolfe} satisfies 
$\sigma \langle x_n - T(x_n),d_n \rangle \leq \langle x_{n+1} - T(x_{n+1}),d_n \rangle \leq - \sigma \langle x_n - T(x_n),d_n \rangle$ and \eqref{inequality:1} holds for some $n$, it is found that 
\begin{align*}
-1 + \frac{\left\langle x_{n+1} - T\left(x_{n+1} \right), d_n \right\rangle}{\left\| x_{n} - T \left( x_{n} \right)  \right\|^2}
&\geq -1 + \sigma \frac{\left\langle x_{n} - T\left(x_{n} \right), d_n \right\rangle}{\left\| x_{n} - T \left( x_{n} \right)  \right\|^2}\\
&\geq -1 - \sigma \sum_{j=0}^n \sigma^j
= - \sum_{j=0}^{n+1} \sigma^j
\end{align*}
and
\begin{align*}
-1 + \frac{\left\langle x_{n+1} - T\left(x_{n+1} \right), d_n \right\rangle}{\left\| x_{n} - T \left( x_{n} \right)  \right\|^2}
&\leq -1 - \sigma \frac{\left\langle x_{n} - T\left(x_{n} \right), d_n \right\rangle}{\left\| x_{n} - T \left( x_{n} \right)  \right\|^2}\\
&\leq -1 + \sigma \sum_{j=0}^n \sigma^j
= -2 + \sum_{j=0}^{n+1} \sigma^j.
\end{align*}
Hence, 
\begin{align*}
- \sum_{j=0}^{n+1} \sigma^j \leq \frac{\left\langle x_{n+1} - T\left(x_{n+1} \right), d_{n+1} \right\rangle}
{\left\| x_{n+1} - T\left(x_{n+1} \right)
\right\|^2} \leq -2 + \sum_{j=0}^{n+1} \sigma^j.
\end{align*} 
A discussion similar to the one for obtaining $\langle x_n- T(x_n), d_n \rangle < 0$ guarantees that 
$\langle x_{n+1} - T(x_{n+1}), d_{n+1} \rangle < 0$ holds.
Induction thus shows that \eqref{inequality:1} and $\langle x_n- T(x_n), d_n \rangle < 0$ hold for all $n\in\mathbb{N}$.

(ii)
Assume that $\liminf_{n\to\infty} \| x_n - T(x_n) \| > 0$.
A discussion similar to the one in the proof of Lemma \ref{lem:DY}(ii) ensures the existence of $\varepsilon > 0$
such that  
$\|x_n - T(x_n)\| \geq \varepsilon$ for all $n \in \mathbb{N}$.
From \eqref{s_Wolfe} and \eqref{inequality:1}, we have for all $n\in\mathbb{N}$,
\begin{align*}
\left| \left\langle x_{n+1} - T \left( x_{n+1} \right), d_n \right\rangle \right|
< - \sigma \left\langle x_{n} - T \left( x_{n} \right), d_n \right\rangle
\leq \sum_{j=1}^{n+1} \sigma^{j} \left\| x_{n} - T\left(x_{n} \right) \right\|^2,
\end{align*}
which, together with $\sum_{j=1}^{n+1} \sigma^{j} <  \sum_{j=1}^{\infty} \sigma^{j} = \sigma/(1 - \sigma)$
and $\beta_n^{\mathrm{FR}} := \| x_{n+1} - T ( x_{n+1} ) \|^2/\| x_n - T(x_n) \|^2$ $(n\in\mathbb{N})$,
implies that, for all $n\in\mathbb{N}$,
\begin{align*}
\beta_n^{\mathrm{FR}} \left| \left\langle x_{n+1} - T \left( x_{n+1} \right), d_n \right\rangle \right| 
< \frac{\sigma}{1 - \sigma} \left\| x_{n+1} - T\left(x_{n+1} \right) \right\|^2.
\end{align*}
Accordingly, from the definition of $d_{n+1} := - (x_{n+1} - T(x_{n+1})) + \beta_n^{\mathrm{FR}} d_n$,
we find that,
for all $n\in\mathbb{N}$, 
\begin{align*}
\left\| d_{n+1} \right\|^2
&= \left\| \beta_n^{\mathrm{FR}} d_n  - \left(x_{n+1} - T \left(x_{n+1} \right) \right) \right\|^2\\
&= \beta_n^{\mathrm{FR}^2}  \left\| d_n \right\|^2 
   - 2 \beta_n^{\mathrm{FR}} \left\langle d_n, x_{n+1} - T \left(x_{n+1} \right)\right\rangle + 
   \left\| x_{n+1} - T \left(x_{n+1} \right) \right\|^2\\
&\leq \frac{\left\| x_{n+1} - T\left(x_{n+1} \right)\right\|^4}{\left\| x_{n} - T\left(x_{n} \right)\right\|^4}  \left\| d_n \right\|^2
    + \left( \frac{2\sigma}{1 - \sigma} + 1 \right)  \left\| x_{n+1} - T \left(x_{n+1} \right) \right\|^2,  
\end{align*}
which means that, for all $n\in\mathbb{N}$,
\begin{align*}
\frac{\left\| d_{n+1} \right\|^2}{\left\| x_{n+1} - T\left(x_{n+1} \right)\right\|^4}
\leq \frac{\left\| d_{n} \right\|^2}{\left\| x_{n} - T\left(x_{n} \right)\right\|^4} 
    + \frac{1+\sigma}{1-\sigma} \frac{1}{\left\| x_{n+1} - T \left(x_{n+1} \right) \right\|^2}.
\end{align*}
The sum of this inequality from $n=0$ to $n=N \in\mathbb{N}$ and $d_0 := - (x_0 - T(x_0))$ ensure that,
for all $N\in \mathbb{N}$,
\begin{align*}
\frac{\left\| d_{N+1} \right\|^2}{\left\| x_{N+1} - T\left(x_{N+1} \right)\right\|^4}
\leq \frac{1}{\left\| x_{0} - T \left(x_{0} \right) \right\|^2} 
      + \frac{1+\sigma}{1-\sigma} \sum_{k=1}^{N+1} \frac{1}{\left\| x_{k} - T \left(x_{k} \right) \right\|^2}.
\end{align*}
From $\|x_n - T(x_n)\| \geq \varepsilon$ ($n \in \mathbb{N}$), for all $N\in\mathbb{N}$,
\begin{align*}
\frac{\left\| d_{N+1} \right\|^2}{\left\| x_{N+1} - T\left(x_{N+1} \right)\right\|^4}
\leq \frac{1}{\varepsilon^2} + \frac{1+\sigma}{1-\sigma}\frac{N+1}{\varepsilon^2}
= \frac{\left( 1 + \sigma \right) N + 2}{\varepsilon^2 \left( 1 - \sigma  \right)}.
\end{align*}
Therefore, from Lemma \ref{lem:FR}(i) guaranteeing that $\|d_n\|\neq 0$ $(n\in\mathbb{N})$ and $\sum_{k=1}^\infty \varepsilon^2 ( 1 - \sigma)/( ( 1 + \sigma) (k -1) + 2) = \infty$,
it is found that
\begin{align*}
\sum_{k=1}^\infty \frac{\left\| x_{k} - T\left(x_{k} \right)\right\|^4}{\left\| d_{k} \right\|^2} = \infty.
\end{align*}
Meanwhile, since \eqref{inequality:1} guarantees that $\langle x_n - T(x_n), d_n \rangle 
\leq (-2 + \sum_{j=0}^n \sigma^j ) \| x_n - T(x_n) \|^2 < (-(1-2 \sigma)/(1-\sigma)) \| x_n - T(x_n) \|^2$ $(n\in\mathbb{N})$,
Lemma \ref{Zou} and Lemma \ref{lem:FR}(i) lead to the deduction that 
\begin{align*}
\infty > \sum_{k=0}^\infty \left( \frac{\left\langle x_k - T\left(x_k\right), d_k   \right\rangle}
{\left\|d_k \right\|} \right)^2
\geq \left( \frac{1- 2 \sigma}{1-\sigma} \right)^2 \sum_{k=0}^\infty \frac{\left\| x_{k} - T\left(x_{k} \right)\right\|^4}{\left\|d_k \right\|^2} = \infty,
\end{align*}
which is a contradiction.
Therefore,  $\liminf_{n\to\infty} \| x_n - T(x_n) \| = 0$.

(iii)
A discussion similar to the one in the proof of Lemma \ref{lem:DY}(iii) leads to Lemma \ref{lem:FR}(iii).
This completes the proof.
\end{proof}

\subsection{Algorithm \ref{algo:1} with $\beta_n = \beta_n^{\mathrm{PRP}+}$}\label{subsec:PRP}
It is  well known that the convergence of the nonlinear conjugate gradient method with $\beta_n^{\mathrm{PRP}}$ defined as in \eqref{beta} for a general nonlinear function is uncertain \cite[Section 5]{hager2006}. To guarantee the convergence of the PRP method for unconstrained optimization, the following modification of $\beta_n^{\mathrm{PRP}}$ was presented in \cite{powell1984}: for $\beta_n^{\mathrm{PRP}}$ defined as in \eqref{beta}, $\beta_n^{\mathrm{PRP+}} := \max \{ \beta_n^{\mathrm{PRP}}, 0 \}$. On the basis of the idea behind this modification, this subsection considers Algorithm \ref{algo:1} with 
$\beta_n^{\mathrm{PRP}+}$ defined as in \eqref{Beta}.

\begin{thm}\label{thm:4}
Suppose that $(x_n)_{n\in\mathbb{N}}$ and $(d_n)_{n\in\mathbb{N}}$ are the sequences generated by Algorithm \ref{algo:1} 
with $\beta_n = \beta_n^{\mathrm{PRP+}}$ $(n\in\mathbb{N})$ and there exists $c > 0$ such that $\langle x_n - T(x_n), d_n \rangle \leq -c \| x_n - T(x_n) \|^2$ for all $n\in \mathbb{N}$.
If $(x_n)_{n\in\mathbb{N}}$ is bounded,
then $(x_n)_{n\in\mathbb{N}}$ either terminates at a fixed point of $T$ or 
\begin{align*}
\lim_{n\to \infty} \left\|x_n - T \left(x_n \right) \right\| = 0.
\end{align*}
\end{thm}

\subsubsection{Proof of Theorem \ref{thm:4}}
It can be assumed that $\| x_n - T (x_n) \| \neq 0$ holds for all $n\in\mathbb{N}$.
Let us first show the following lemma by referring to the proof of \cite[Lemma 4.1]{gilbert1992}. 

\begin{lem}\label{Gil_1}
Let $(x_n)_{n\in \mathbb{N}}$ and $(d_n)_{n\in \mathbb{N}}$ be the sequences generated by Algorithm \ref{algo:1} with $\beta_n \geq 0$
$(n\in\mathbb{N})$ and assume that there exists $c > 0$ such that 
$\langle x_n - T(x_n), d_n \rangle \leq - c \|x_n - T(x_n)\|^2$ for all $n\in\mathbb{N}$.
If there exists $\varepsilon > 0$ such that $\|x_n - T(x_n)\| \geq \varepsilon$ for all $n\in \mathbb{N}$,
then
$\sum_{n=0}^\infty \| u_{n+1} - u_n \|^2< \infty$, 
where $u_n := d_n/\|d_n\|$ $(n\in\mathbb{N})$.
\end{lem}

\begin{proof}
Assuming $\| x_n - T (x_n) \| \geq \varepsilon$ and $\langle x_n - T(x_n), d_n \rangle \leq - c \|x_n - T(x_n)\|^2$
$(n\in \mathbb{N})$, $\| d_n \|\neq 0$ holds for all $n\in \mathbb{N}$.
Define $r_n := - (x_n - T(x_n))/\|d_n\|$ and $\delta_n := \beta_n \|d_n\|/\| d_{n+1} \|$ $(n\in\mathbb{N})$.
From $\delta_n u_n = \beta_n d_n /\| d_{n+1}\|$ and $d_{n+1} = - (x_{n+1} - T(x_{n+1})) + \beta_n d_n$ $(n\in\mathbb{N})$,
we have for all $n\in\mathbb{N}$,
\begin{align*}
u_{n+1} = - r_{n+1} + \delta_n u_n,
\end{align*}
which, together with $\| u_{n+1} - \delta_n u_n \|^2 = \|u_{n+1}\|^2 -2 \delta_n \langle u_{n+1}, u_n \rangle + \delta_n^2 \|u_n\|^2
= \|u_{n}\|^2 -2 \delta_n \langle u_{n}, u_{n+1} \rangle + \delta_n^2 \|u_{n+1} \|^2 = 
\| u_n - \delta_n u_{n+1} \|^2$
$(n\in\mathbb{N})$, implies that, for all $n\in\mathbb{N}$,
\begin{align*}
\left\| r_{n+1} \right\| = \left\| u_{n+1} - \delta_n u_n \right\| = \left\| u_{n} - \delta_n u_{n+1} \right\|.
\end{align*}
Accordingly, the condition $\beta_n \geq 0$ $(n\in\mathbb{N})$ and the triangle inequality mean that, for all $n\in\mathbb{N}$,
\begin{align}\label{un}
\begin{split}
\left\| u_{n+1} - u_n \right\| 
&\leq \left(1+ \delta_n \right) \left\| u_{n+1} - u_n \right\|\\
&\leq \left\| u_{n+1} - \delta_n u_n \right\| + \left\| u_{n} - \delta_n u_{n+1} \right\|\\
&= 2 \left\| r_{n+1} \right\|.
\end{split}
\end{align}
From Lemma \ref{Zou}, $\langle x_n - T(x_n), d_n \rangle \leq - c \|x_n - T(x_n)\|^2$
$(n\in \mathbb{N})$, the definition of $r_n$, and $\| x_n - T(x_n) \| \geq \varepsilon$ $(n\in\mathbb{N})$,
we have
\begin{align*}
\infty > \sum_{n=0}^\infty \left( \frac{\left\langle x_n- T\left(x_n\right), d_n \right\rangle}{\left\|d_n\right\|}  \right)^2
\geq c^2 \sum_{n=0}^\infty \frac{\left\|x_n - T \left(x_n \right) \right\|^4}{\left\|d_n\right\|^2}
\geq c^2 \varepsilon^2 \sum_{n=0}^\infty \left\| r_{n} \right\|^2,
\end{align*}
which, together with \eqref{un}, completes the proof.
\end{proof}

The following property, referred to as Property ($\star$), is a result of modifying \cite[Property (*)]{gilbert1992} to conform to Problem \eqref{prob:1}.

Property ($\star$). Suppose that there exist positive constants $\gamma$ and $\bar{\gamma}$ such that $\gamma \leq \| x_n - T(x_n) \| \leq \bar{\gamma}$ for all $n\in \mathbb{N}$. Then Property ($\star$) holds if $b > 1$ and $\lambda > 0$ exist such that, for all $n\in \mathbb{N}$, 
\begin{align*}
\left|\beta_n \right| \leq b 
\text{ and } 
\left\| x_{n+1} - x_n \right\| \leq \lambda \text{ implies }  \left|\beta_n \right| \leq \frac{1}{2b}.
\end{align*}

The proof of the following lemma can be omitted since it is similar to the proof of \cite[Lemma 4.2]{gilbert1992}.

\begin{lem}\label{Gil_2}
Let $(x_n)_{n\in \mathbb{N}}$ and $(d_n)_{n\in \mathbb{N}}$ be the sequences generated by Algorithm \ref{algo:1} and assume that there exist $c > 0$ 
and $\gamma > 0$ such that 
$\langle x_n - T(x_n), d_n \rangle \leq - c \|x_n - T(x_n)\|^2$
and
$\|x_n - T(x_n)\| \geq \gamma$ for all $n\in \mathbb{N}$.
Suppose also that Property $(\star)$ holds.
Then there exists $\lambda > 0$ such that, for all $\Delta \in \mathbb{N} \backslash \{0\}$
and for any index $k_0$, there is $k \geq k_0$ such that $| \mathcal{K}_{k,\Delta}^\lambda | > \Delta/2$,
where 
$\mathcal{K}_{k,\Delta}^\lambda := \{ i\in \mathbb{N} \backslash \{0\} \colon  
      k \leq i \leq k + \Delta -1, \| x_{i} - x_{i-1} \| > \lambda \}$ 
      $(k\in \mathbb{N}, \Delta \in \mathbb{N} \backslash \{0\},\lambda > 0)$ 
      and $|\mathcal{K}_{k,\Delta}^\lambda|$ stands for the number of elements of $\mathcal{K}_{k,\Delta}^\lambda$.
\end{lem}

The following can be proven by referring to the proof of \cite[Theorem 4.3]{gilbert1992}.

\begin{lem}\label{Gil_3}
Let $(x_n)_{n\in \mathbb{N}}$ be the sequence generated by Algorithm \ref{algo:1} with $\beta_n \geq 0$
$(n\in\mathbb{N})$ and assume that there exists $c > 0$ such that 
$\langle x_n - T(x_n), d_n \rangle \leq - c \|x_n - T(x_n)\|^2$ for all $n\in\mathbb{N}$
and Property $(\star)$ holds.
If $(x_n)_{n\in\mathbb{N}}$ is bounded,  
$\liminf_{n\to \infty} \|x_n - T (x_n ) \| = 0$.
\end{lem}

\begin{proof}
Assuming that $\liminf_{n\to \infty} \|x_n - T (x_n ) \| > 0$, there exists $\gamma > 0$ such that $\| x_n - T(x_n) \| \geq \gamma$ for all $n\in\mathbb{N}$. Since $c> 0$ exists such that $\langle x_n - T(x_n), d_n \rangle \leq - c \|x_n - T(x_n)\|^2$ $(n\in\mathbb{N})$, $\| d_n \| \neq 0$ $(n\in\mathbb{N})$ holds. Moreover, the nonexpansivity of $T$ ensures that, for all $x\in \mathrm{Fix}(T)$, $\| T (x_n ) - x \| \leq \| x_n -x \|$, and this, together with the boundedness of $(x_n)_{n\in\mathbb{N}}$, implies the boundedness of $(T(x_n))_{n\in\mathbb{N}}$. Accordingly, $\bar{\gamma} > 0$ exists such that $\| x_n - T(x_n) \| \leq \bar{\gamma}$ $(n\in\mathbb{N})$. The definition of $x_n$ implies that, for all $n\geq 1$,
\begin{align*}
x_n - x_{n-1} = \alpha_{n-1} d_{n-1} = \alpha_{n-1} \left\|d_{n-1} \right\| u_{n-1} 
= \left\|x_n - x_{n-1}\right\| u_{n-1},
\end{align*}
where $u_n := d_n/\|d_n\|$ $(n\in\mathbb{N})$.
Hence, for all $l, k \in \mathbb{N}$ with $l \geq k > 0$,
\begin{align*}
x_l - x_{k-1} 
= \sum_{i=k}^l \left( x_i - x_{i-1} \right) = \sum_{i=k}^l \left\|x_i - x_{i-1}\right\| u_{i-1},
\end{align*}
which implies that
\begin{align*}
\sum_{i=k}^l \left\|x_i - x_{i-1}\right\| u_{k-1} 
= x_l - x_{k-1} - \sum_{i=k}^l  \left\|x_i - x_{i-1}\right\| \left(u_{i-1} - u_{k-1} \right).
\end{align*}
From $\| u_n \| = 1$ $(n\in\mathbb{N})$ and the triangle inequality, for all $l, k \in \mathbb{N}$ with $l \geq k > 0$,
$\sum_{i=k}^l \|x_i - x_{i-1} \| \leq \| x_l - x_{k-1} \| + 
\sum_{i=k}^l  \|x_i - x_{i-1} \| \| u_{i-1} - u_{k-1} \|$.
Since the boundedness of $(x_n)_{n\in\mathbb{N}}$ means there is $M > 0$ satisfying 
$\| x_{n+1} - x_n \| \leq M$ $(n\in \mathbb{N})$,
we find that, for all $l, k \in \mathbb{N}$ with $l \geq k > 0$,
\begin{align}\label{inequality:Gil}
\sum_{i=k}^l \left\|x_i - x_{i-1} \right\| \leq M 
   + \sum_{i=k}^l  \left\|x_i - x_{i-1} \right\| \left\| u_{i-1} - u_{k-1} \right\|.
\end{align}
Let $\lambda > 0$ be as given by Lemma \ref{Gil_2} and define $\Delta := \lceil 4M/\lambda \rceil$,
where $\lceil \cdot \rceil$ denotes the ceiling operator.
From Lemma \ref{Gil_1}, an index $k_0$ can be chosen such that 
$\sum_{i=k_0}^\infty \| u_i - u_{i-1} \|^2 \leq 1/(4 \Delta)$.
Accordingly, Lemma \ref{Gil_2} guarantees the existence of $k \geq k_0$ such that 
$| \mathcal{K}_{k,\Delta}^\lambda | > \Delta/2$.
Since the Cauchy-Schwarz inequality implies that $(\sum_{i=1}^m a_i)^2 \leq m \sum_{i=1}^m a_i^2$
$(m \geq 1, a_i \in \mathbb{R}, i=1,2,\ldots,m)$,
we have,
for all $i\in [k,k+\Delta -1]$,
\begin{align*}
\left\| u_{i-1} - u_{k-1} \right\|^2
\leq \left( \sum_{j=k}^{i-1} \left\| u_{j} - u_{j-1} \right\| \right)^2
\leq \left( i - k \right) \sum_{j=k}^{i-1} \left\| u_{j} - u_{j-1} \right\|^2
\leq 
\frac{1}{4}.
\end{align*}
Putting $l:= k+\Delta -1$, \eqref{inequality:Gil} ensures that 
\begin{align*}
M \geq \frac{1}{2} \sum_{i=k}^{k+\Delta -1} \left\| x_i - x_{i-1} \right\| 
> \frac{\lambda}{2} \left| \mathcal{K}_{k,\Delta}^\lambda \right| > \frac{\lambda \Delta}{4},
\end{align*}
which implies that $\Delta < 4M/\lambda$. 
This contradicts $\Delta := \lceil 4M/\lambda \rceil$.
Therefore, $\liminf_{n\to \infty} \|x_n - T (x_n ) \| = 0$.
\end{proof}

Now we are in the position to prove Theorem \ref{thm:4}. 

\begin{proof}
The condition $\beta_n^{\mathrm{PRP+}} \geq 0$ holds for all $n\in\mathbb{N}$.
Suppose that positive constants $\gamma$ and $\bar{\gamma}$ exist such that $\gamma \leq \|x_n - T(x_n)\| \leq \bar{\gamma}$
$(n\in\mathbb{N})$ and define $b:= 2\bar{\gamma}^2/\gamma^2$ and $\lambda := \gamma^2/(4\bar{\gamma} b)$.
The definition of $\beta_n^{\mathrm{PRP+}}$ and the Cauchy-Schwarz inequality mean that, for all $n\in\mathbb{N}$,
\begin{align*}
\left|\beta_n^{\mathrm{PRP+}} \right| 
\leq  
\frac{\left|\left\langle x_{n+1} - T \left(x_{n+1}\right), y_n \right\rangle \right|}{\left\| x_{n} - T \left(x_{n}\right) \right\|^2}
\leq 
\frac{\left\| x_{n+1} - T \left(x_{n+1}\right) \right\| \left\|y_n \right\|}{\left\| x_{n} - T \left(x_{n}\right) \right\|^2}
\leq \frac{2 \bar{\gamma}^2}{\gamma^2} = b,
\end{align*}
where the third inequality comes from 
$\|y_n\| \leq \|x_{n+1} - T(x_{n+1})\| + \| x_n - T(x_n)\| \leq 2 \bar{\gamma}$ and 
$\gamma \leq \|x_n - T(x_n)\| \leq \bar{\gamma}$ $(n\in\mathbb{N})$.
When $\| x_{n+1} - x_n \| \leq \lambda$ $(n\in\mathbb{N})$, the triangle inequality and the nonexpansivity of $T$ imply that 
$\|y_n\| \leq \|x_{n+1} - x_{n}\| + \| T(x_{n}) - T(x_{n+1})\| \leq 2 \| x_{n+1} - x_n \| \leq 2 \lambda$ $(n\in\mathbb{N})$.
Therefore, for all $n\in\mathbb{N}$,
\begin{align*}
\left|\beta_n^{\mathrm{PRP+}} \right| 
\leq 
\frac{\bar{\gamma} \left\|y_n \right\|}{\left\| x_{n} - T \left(x_{n}\right) \right\|^2}
\leq \frac{2 \lambda \bar{\gamma}}{\gamma^2} = \frac{1}{2b},
\end{align*}
which implies that Property $(\star)$ holds.
Lemma \ref{Gil_3} thus guarantees that $\liminf_{n\to\infty} \| x_n - T(x_n) \| = 0$ holds.
A discussion in the same manner as in the proof of Lemma \ref{lem:DY}(iii) leads to $\lim_{n\to\infty} \| x_n - T(x_n) \| = 0$.
This completes the proof.
\end{proof}

\subsection{Algorithm \ref{algo:1} with $\beta_n = \beta_n^{\mathrm{HS}+}$}\label{subsec:HS}
The convergence properties of the nonlinear conjugate gradient method with $\beta_n^{\mathrm{HS}}$ defined as in \eqref{beta} are similar to those with $\beta_n^{\mathrm{PRP}}$ defined as in \eqref{beta} \cite[Section 5]{hager2006}. 
On the basis of this fact and the modification of $\beta_n^{\mathrm{PRP}}$ in Subsection \ref{subsec:PRP}, 
this subsection considers Algorithm \ref{algo:1} with $\beta_n^{\mathrm{HS}+}$ defined by \eqref{Beta}.
%

Lemma \ref{Gil_3} leads to the following.

\begin{thm}\label{thm:5}
Suppose that $(x_n)_{n\in\mathbb{N}}$ and $(d_n)_{n\in\mathbb{N}}$ are the sequences generated by Algorithm \ref{algo:1} 
with $\beta_n = \beta_n^{\mathrm{HS+}}$ $(n\in\mathbb{N})$ and there exists $c > 0$ such that $\langle x_n - T(x_n), d_n \rangle \leq -c \| x_n - T(x_n) \|^2$ for all $n\in \mathbb{N}$.
If $(x_n)_{n\in\mathbb{N}}$ is bounded,
then $(x_n)_{n\in\mathbb{N}}$ either terminates at a fixed point of $T$ or 
\begin{align*}
\lim_{n\to \infty} \left\|x_n - T \left(x_n \right) \right\| = 0.
\end{align*}
\end{thm}

\begin{proof}
When $m\in \mathbb{N}$ exists such that $\|x_m - T(x_m) \| =0$, Theorem \ref{thm:5} holds.
Let us consider the case where $\| x_n - T(x_n) \| \neq 0$ for all $n\in\mathbb{N}$.
Suppose that $\gamma, \bar{\gamma} > 0$ exist such that $\gamma \leq \| x_n - T(x_n) \| \leq \bar{\gamma}$ $(n\in\mathbb{N})$
and define $b:= 2\bar{\gamma}^2/((1-\sigma)c\gamma^2)$ and $\lambda := (1-\sigma)c \gamma^2/(4\bar{\gamma}b)$.
Then \eqref{Wolfe} implies that, for all $n\in \mathbb{N}$,
\begin{align*}
\left\langle d_n, y_n \right\rangle
&= \left\langle d_n, x_{n+1} - T \left( x_{n+1} \right) \right\rangle - \left\langle d_n, x_{n} - T \left( x_{n} \right) \right\rangle\\
&\geq - \left( 1 - \sigma \right)  \left\langle d_n, x_{n} - T \left( x_{n} \right) \right\rangle,
\end{align*}
which, together with the existence of $c, \gamma > 0$ such that $\langle x_n - T(x_n), d_n \rangle \leq -c \| x_n - T(x_n) \|^2$, and $\gamma \leq \| x_n - T(x_n) \|$ $(n\in \mathbb{N})$, implies that, for all $n\in\mathbb{N}$,
\begin{align*}
\left\langle d_n, y_n \right\rangle \geq \left( 1 - \sigma \right) c \left\| x_n - T \left(x_n\right) \right\|^2
\geq \left( 1 - \sigma \right) c \gamma^2 > 0.
\end{align*}
This means Algorithm \ref{algo:1} with $\beta_n = \beta_n^{\mathrm{HS+}}$ is well-defined.
From $\|x_n - T(x_n)\| \leq \bar{\gamma}$ $(n\in\mathbb{N})$ and the definition of $y_n$,
we have,
for all $n\in\mathbb{N}$,
\begin{align*}
\left| \beta_n^{\mathrm{HS+}} \right| 
\leq \frac{\left|\left\langle x_{n+1} - T \left(x_{n+1}\right), y_n \right\rangle \right|}
{\left|\left\langle d_n, y_n \right\rangle \right|}
\leq \frac{2 \bar{\gamma}^2}{\left( 1 - \sigma \right) c \gamma^2} = b.
\end{align*}
When $\| x_{n+1} - x_n \| \leq \lambda$ $(n\in\mathbb{N})$, the triangle inequality and the nonexpansivity of $T$ imply that 
$\|y_n\| \leq \|x_{n+1} - x_{n}\| + \| T(x_{n}) - T(x_{n+1})\| \leq 2 \| x_{n+1} - x_n \| \leq 2 \lambda$ $(n\in\mathbb{N})$.
Therefore, from $\| x_n - T(x_n) \| \leq \bar{\gamma}$ $(n\in\mathbb{N})$, for all $n\in\mathbb{N}$,
\begin{align*}
\left|\beta_n^{\mathrm{HS+}} \right| 
\leq 
\frac{\bar{\gamma} \left\|y_n \right\|}{\left\langle d_n, y_n \right\rangle}
\leq \frac{2 \lambda \bar{\gamma}}{\left( 1 - \sigma \right) c \gamma^2} = \frac{1}{2b},
\end{align*}
which in turn implies that Property $(\star)$ holds.
Lemma \ref{Gil_3} thus ensures that $\liminf_{n\to\infty} \| x_n - T(x_n) \| = 0$ holds.
A discussion similar to the one in the proof of Lemma \ref{lem:DY}(iii) leads to $\lim_{n\to\infty} \| x_n - T(x_n) \| = 0$.
This completes the proof.
\end{proof}

\subsection{Convergence rate analyses of Algorithm \ref{algo:1}}\label{subsec:2.6}
Subsections \ref{subsec:SD}--\ref{subsec:HS} show that Algorithm \ref{algo:1} with formulas \eqref{Beta} satisfies $\lim_{n\to \infty} \| x_n - T(x_n) \| = 0$ under certain assumptions. The next theorem establishes rates of convergence for Algorithm \ref{algo:1} with formulas \eqref{Beta}.
\begin{thm}\label{rate}
\begin{enumerate}
\item[{\em (i)}]
Under the Wolfe-type conditions \eqref{Armijo} and \eqref{Wolfe}, 
Algorithm \ref{algo:1} with $\beta_n = \beta_n^{\mathrm{SD}}$ satisfies, for all $n\in \mathbb{N}$,
\begin{align*}
\left\| x_n - T \left(x_n\right) \right\| 
\leq \frac{\left\| x_0 - T \left(x_0 \right) \right\|}{\sqrt{\delta \sum_{k=0}^n \alpha_k}}.
\end{align*}
\item[{\em (ii)}]
Under the strong Wolfe-type conditions \eqref{Armijo} and \eqref{s_Wolfe}, 
Algorithm \ref{algo:1} with $\beta_n = \beta_n^{\mathrm{DY}}$ satisfies, for all $n\in \mathbb{N}$,
\begin{align*}
\left\| x_n - T \left(x_n\right) \right\| 
\leq \frac{\left\| x_0 - T \left(x_0 \right) \right\|}{\sqrt{\frac{1}{1+\sigma} \delta \sum_{k=0}^n \alpha_k}}.
\end{align*}
\item[{\em (iii)}]
Under the strong Wolfe-type conditions \eqref{Armijo} and \eqref{s_Wolfe}, 
Algorithm \ref{algo:1} with $\beta_n = \beta_n^{\mathrm{FR}}$ satisfies, for all $n\in \mathbb{N}$,
\begin{align*}
\left\| x_n - T \left(x_n\right) \right\| 
\leq \frac{\left\| x_0 - T \left(x_0 \right) \right\|}{\sqrt{\frac{1}{1-\sigma} \delta \sum_{k=0}^n \left( 1-2\sigma + \sigma^k \right) \alpha_k}}.
\end{align*}
\item[{\em (iv)}]
Under the assumptions in Theorem \ref{thm:4},
Algorithm \ref{algo:1} with $\beta_n = \beta_n^{\mathrm{PRP}+}$ satisfies, for all $n\in \mathbb{N}$,
\begin{align*}
\left\| x_n - T \left(x_n\right) \right\| 
\leq \frac{\left\| x_0 - T \left(x_0 \right) \right\|}{\sqrt{c \delta \sum_{k=0}^n \alpha_k}}.
\end{align*}
\item[{\em (v)}]
Under the assumptions in Theorem \ref{thm:5},
Algorithm \ref{algo:1} with $\beta_n = \beta_n^{\mathrm{HS}+}$ satisfies, for all $n\in \mathbb{N}$,
\begin{align*}
\left\| x_n - T \left(x_n\right) \right\| 
\leq \frac{\left\| x_0 - T \left(x_0 \right) \right\|}{\sqrt{c \delta \sum_{k=0}^n \alpha_k}}.
\end{align*}
\end{enumerate}
\end{thm}

\begin{proof}
(i)
From $d_k = - (x_k - T(x_k))$ $(k\in \mathbb{N})$ and \eqref{Armijo}, we have that
$0 \leq \delta \alpha_k \|x_k - T(x_k)\|^2 \leq \| x_{k} - T(x_{k}) \|^2 - \|x_{k+1} - T(x_{k+1})\|^2$ $(k\in \mathbb{N})$.
Summing up this inequality from $k=0$ to $k=n$ guarantees that, for all $n\in \mathbb{N}$,
\begin{align*}
\delta \sum_{k=0}^n \alpha_k \left\|x_k - T\left(x_k\right) \right\|^2 
\leq \left\| x_{0} - T \left(x_{0} \right) \right\|^2 - \left\|x_{n+1} - T \left(x_{n+1}\right) \right\|^2 
\leq \left\| x_{0} - T \left(x_{0} \right) \right\|^2,
\end{align*}
which, together with the monotone decreasing property of $(\| x_n - T(x_n) \|^2)_{n\in \mathbb{N}}$, implies that, for all $n\in \mathbb{N}$,
\begin{align*}
\delta \left\|x_n - T\left(x_n\right) \right\|^2 \sum_{k=0}^n \alpha_k  
\leq \left\| x_{0} - T \left(x_{0} \right) \right\|^2.
\end{align*}
This completes the proof.

(ii)
Condition \eqref{s_Wolfe} and Lemma \ref{lem:DY}(i) ensure that 
$- \sigma \leq \langle x_{k+1} - T(x_{k+1}), d_k \rangle/\langle x_{k} - T(x_{k}), d_k \rangle \leq \sigma$ $(k\in \mathbb{N})$.
Accordingly, \eqref{DY1} means that, for all $k\in \mathbb{N}$,
\begin{align*}
\left\langle x_{k+1} - T \left(x_{k+1} \right), d_{k+1} \right\rangle 
&= \frac{\left\langle x_{k} - T \left(x_{k} \right), d_{k} \right\rangle}{\left\langle d_{k}, \left(x_{k+1} - T \left(x_{k+1} \right) \right) - 
  \left( x_{k} - T \left(x_{k} \right) \right) \right\rangle} 
  \left\| x_{k+1} - T \left(x_{k+1} \right) \right\|^2\\
&= \left(\frac{\langle x_{k+1} - T(x_{k+1}), d_k \rangle}{\langle x_{k} - T(x_{k}), d_k \rangle} -1 \right)^{-1}  
  \left\| x_{k+1} - T \left(x_{k+1} \right) \right\|^2\\
&\leq - \frac{1}{1+\sigma} \left\| x_{k+1} - T \left(x_{k+1} \right) \right\|^2.   
\end{align*}
Hence, \eqref{Armijo} implies that, for all $k\in \mathbb{N}$,
\begin{align*}
\left\|  x_{k+1} - T \left(x_{k+1} \right) \right\|^2 - \left\|  x_{k} - T \left(x_{k} \right) \right\|^2
\leq - \frac{1}{1+\sigma} \delta \alpha_k \left\| x_{k} - T \left(x_{k} \right) \right\|^2.
\end{align*}
Summing up this inequality from $k=0$ to $k=n$ and the monotone decreasing property of $(\| x_n - T(x_n) \|^2)_{n\in \mathbb{N}}$ 
ensure that, for all $n\in\mathbb{N}$,
\begin{align*}
\frac{1}{1+\sigma} \delta \left\|x_n - T\left(x_n\right) \right\|^2 \sum_{k=0}^n \alpha_k  
\leq \left\| x_{0} - T \left(x_{0} \right) \right\|^2,
\end{align*}
which completes the proof.

(iii)
Inequality \eqref{inequality:1} guarantees that, for all $k\in \mathbb{N}$,
\begin{align*}
\left\langle x_{k} - T \left(x_{k} \right), d_{k} \right\rangle
\leq \left(-2 + \sum_{j=0}^k \sigma^j \right) \left\| x_{k} - T \left(x_{k} \right) \right\|^2
= - \frac{1-2\sigma + \sigma^k}{1 - \sigma} \left\| x_{k} - T \left(x_{k} \right) \right\|^2,
\end{align*}
which, together with \eqref{Armijo}, implies that, for all $k\in \mathbb{N}$,
\begin{align*}
\left\|  x_{k+1} - T \left(x_{k+1} \right) \right\|^2 - \left\|  x_{k} - T \left(x_{k} \right) \right\|^2
\leq - \frac{1-2\sigma + \sigma^k}{1-\sigma} \delta \alpha_k \left\| x_{k} - T \left(x_{k} \right) \right\|^2.
\end{align*}
Summing up this inequality from $k=0$ to $k=n$ and the monotone decreasing property of $(\| x_n - T(x_n) \|^2)_{n\in \mathbb{N}}$ 
ensure that, for all $n\in\mathbb{N}$,
\begin{align*}
\frac{1}{1-\sigma} \delta \left\|x_n - T\left(x_n\right) \right\|^2 \sum_{k=0}^n \left(1-2\sigma + \sigma^k \right) \alpha_k  
\leq \left\| x_{0} - T \left(x_{0} \right) \right\|^2,
\end{align*}
which completes the proof.

(iv), (v)
Since there exists $c > 0$ such that $\langle x_k - T(x_k), d_k \rangle \leq  -c \| x_k - T(x_k)\|^2$ for all $k\in \mathbb{N}$,
we have from \eqref{Armijo} and the  monotone decreasing property of $(\| x_n - T(x_n) \|^2)_{n\in \mathbb{N}}$ that, for all $n\in \mathbb{N}$,
\begin{align*}
c \delta \left\| x_{n} - T \left(x_{n} \right) \right\|^2 \sum_{k=0}^n \alpha_k 
\leq
c \delta \sum_{k=0}^n \alpha_k \left\| x_{k} - T \left(x_{k} \right) \right\|^2
\leq 
\left\|  x_{0} - T \left(x_{0} \right) \right\|^2.
\end{align*}
This concludes the proof.
\end{proof}

The conventional Krasnosel'ski\u\i-Mann algorithm \eqref{kra} with a step size sequence $(\alpha_n)_{n\in\mathbb{N}}$ obeying \eqref{Dunn} satisfies the following inequality \cite[Propositions 10 and 11]{cominetti2014}:
\begin{align*}
\left\| x_n - T \left(x_n\right) \right\| 
\leq \frac{\mathrm{d}\left(x_0, \mathrm{Fix} \left(T \right) \right)}{\sqrt{\sum_{k=0}^n \alpha_k \left(1-\alpha_k \right)}}
\text{ } \left(n\in\mathbb{N} \right),
\end{align*}
where $\mathrm{d}(x_0, \mathrm{Fix} (T)) := \min_{x\in \mathrm{Fix}(T)} \| x_0 - x \|$. When $\alpha_n$ $(n\in \mathbb{N})$ is a constant in the range of (0,1), which is the most tractable choice of step size satisfying \eqref{Dunn}, the Krasnosel'ski\u\i-Mann algorithm \eqref{kra} has the rate of convergence, 
\begin{align}\label{dunn_rate}
\left\| x_n - T \left(x_n\right) \right\| = O \left( \frac{1}{\sqrt{n+1}} \right).
\end{align}
Meanwhile, according to Theorem 5 in \cite{mag2004}, Algorithm \eqref{kra} with $(\alpha_n)_{n\in\mathbb{N}}$ satisfying the Armijo-type condition \eqref{armijo0} satisfies, for all $n\in \mathbb{N}$,
\begin{align}\label{mag_rate}
\left\| x_n - T \left(x_n\right) \right\| 
\leq \frac{\left\| x_0 - T \left(x_0 \right) \right\|}{\sqrt{\beta \sum_{k=0}^n \left( \alpha_k - \frac{1}{2} \right)^2}}.
\end{align}

In general, the step sizes satisfying \eqref{Dunn} do not coincide with those satisfying the Armijo-type condition \eqref{armijo0} or the Wolfe-type conditions \eqref{Armijo} and \eqref{Wolfe}. This is because the line search methods based on the Armijo-type conditions \eqref{armijo0} and \eqref{Armijo} determine step sizes at each iteration $n$ so as to satisfy $\| x_{n+1} - T(x_{n+1}) \| < \|x_n - T(x_n)\|$, while the constant step sizes satisfying \eqref{Dunn} do not change at each iteration. Accordingly, it would be difficult to evaluate the efficiency of these algorithms by using only the theoretical convergence rates in \eqref{dunn_rate}, \eqref{mag_rate}, and Theorem \ref{rate}. To verify whether Algorithm \ref{algo:1} with the convergence rates in Theorem \ref{rate} converges faster than the previous algorithms \cite[Propositions 10 and 11]{cominetti2014}, \cite[Theorem 5]{mag2004} with convergence rates \eqref{dunn_rate} and \eqref{mag_rate}, the next section numerically compares their abilities to solve concrete constrained smooth convex optimization problems.

\section{Application of Algorithm \ref{algo:1} to constrained smooth convex optimization}\label{sec:3}
This section considers the following problem:
\begin{align}\label{csco}
\text{Minimize } f \left(x \right) \text{ subject to } x\in C,
\end{align}
where $f \colon \mathbb{R}^d \to \mathbb{R}$ is convex, $\nabla f \colon \mathbb{R}^d \to \mathbb{R}^d$ is Lipschitz continuous 
with a constant $L$, and 
$C \subset \mathbb{R}^d$ is a nonempty, closed convex set onto which the metric projection $P_C$ can be efficiently computed.
\subsection{Experimental conditions and fixed point and line search algorithms used in the experiment}\label{subsec:3.0}
Problem \eqref{csco} can be solved by using the conventional Krasnosel'ski\u\i-Mann algorithm \eqref{kra} 
with a nonexpansive mapping $T := P_C (\mathrm{Id} - \lambda \nabla f)$ satisfying $\mathrm{Fix}(T) = \argmin_{x\in C} f(x)$, where $\lambda \in (0,2/L]$ \cite[Proposition 2.3]{iiduka_JOTA}.
It is represented as follows:  
\begin{align}\label{previous}
x_{n+1} = x_n + \alpha_n \left( P_C \left(x_n - \lambda \nabla f \left(x_n \right) \right) - x_n \right),
\end{align}
where $x_0 \in \mathbb{R}^d$ and $(\alpha_n)_{n\in \mathbb{N}}$ is a sequence satisfying \eqref{Dunn} or the Armijo-type condition \eqref{armijo0}.
Algorithm \ref{algo:1} with $T := P_C (\mathrm{Id} - \lambda \nabla f)$ is as follows: 
\begin{align}\label{proposed}
\begin{split}
&x_{n+1} := x_n + \alpha_n d_n,\\
&d_{n+1} := - \left( x_{n+1} -  P_C \left( x_{n+1} - \lambda \nabla f \left(x_{n+1} \right) \right) \right) + \beta_n d_n,
\end{split}
\end{align}
where $x_0, d_0 := -(x_0- P_C(x_0 - \lambda \nabla f (x_0))) \in \mathbb{R}^d$,
$(\alpha_n)_{n\in\mathbb{N}}$ is a sequence satisfying the Wolfe-type conditions \eqref{Armijo} and \eqref{Wolfe},
and $(\beta_n)_{n\in\mathbb{N}}$ is defined by each of the formulas \eqref{Beta} with 
$T := P_C (\mathrm{Id} - \lambda \nabla f)$ (see also \eqref{formulas}).

The best conventional nonlinear conjugate gradient method for unconstrained smooth nonconvex optimization 
was proposed by Hager and Zhang \cite{HZ2005,HZ2006}, and it uses the HS formula defined as in \eqref{beta}:
\begin{align*}
\beta_n^{\mathrm{HZ}} &:= \frac{1}{\left\langle d_n, y_n \right\rangle} 
\left\langle y_n - 2 d_n \frac{\left\| y_n \right\|^2}{\left\langle d_n, y_n \right\rangle}, \nabla f\left(x_{n+1} \right) \right\rangle\\
&= \beta_n^{\mathrm{HS}} - 2 \frac{\left\| y_n \right\|^2}{\left\langle d_n, y_n \right\rangle}
\frac{\left\langle \nabla f\left(x_{n+1} \right), d_n \right\rangle}{\left\langle d_n, y_n \right\rangle}.
\end{align*}
Replacing $\nabla f$ in the above formula with $\mathrm{Id} - P_C (\mathrm{Id} -\lambda \nabla f)$ leads to the HZ-type formula for Problem \eqref{csco}:
\begin{align}\label{HZ}
&\beta_n^{\mathrm{HZ}} := \beta_n^{\mathrm{HS}} - 2 \frac{\left\| y_n \right\|^2}{\left\langle d_n, y_n \right\rangle}
\frac{\left\langle x_{n+1} - P_C \left(x_{n+1} - \lambda \nabla f \left(x_{n+1} \right) \right), d_n \right\rangle}{\left\langle d_n, y_n \right\rangle},
\end{align}
where $y_n := (x_{n+1} - P_C(x_{n+1} -\lambda \nabla f(x_{n+1}) ) ) - (x_n - P_C(x_n -\lambda \nabla f(x_n)) )$ and 
$\beta_n^{\mathrm{HS}}$ is defined by 
$\beta_n^{\mathrm{HS}} := \langle x_{n+1} - P_C(x_{n+1} - \lambda \nabla f (x_{n+1})), y_n \rangle/\langle d_n, y_n \rangle$.
We tested Algorithm \eqref{proposed} with $\beta_n := \beta_n^{\mathrm{HZ}}$ defined by \eqref{HZ} in order to see how it works on Problem \eqref{csco}.

We used the Virtual Desktop PC at the Ikuta campus of Meiji University. The PC has 8 GB of RAM memory, 1 core Intel Xeon 2.6 GHz CPU, and a Windows 8.1 operating system. The algorithms used in the experiment were written in MATLAB (R2013b), and they are summarized as follows.
\begin{description}
\item[\textbf{SD-1}:]
Algorithm \eqref{previous} with constant step sizes $\alpha_n := 0.5$ $(n\in\mathbb{N})$ \cite[Theorem 5.14]{b-c}. 
\item[\textbf{SD-2}:] 
Algorithm \eqref{previous} with $\alpha_n$ satisfying the Armijo-type condition \eqref{armijo0} when 
$\beta = 0.5$ \cite[Theorems 4 and 8]{mag2004}.
\item[\textbf{SD-3}:]
Algorithm \eqref{proposed} with $\alpha_n$ satisfying the Wolfe-type conditions \eqref{Armijo} and \eqref{Wolfe} 
and $\beta_n := \beta_n^{\mathrm{SD}}$ (Theorem \ref{thm:1}).
\item[\textbf{FR}:]
Algorithm \eqref{proposed} with $\alpha_n$ satisfying the Wolfe-type conditions \eqref{Armijo} and \eqref{Wolfe} 
and $\beta_n := \beta_n^{\mathrm{FR}}$ (Theorem \ref{thm:3}).
\item[\textbf{PRP+}:]
Algorithm \eqref{proposed} with $\alpha_n$ satisfying the Wolfe-type conditions \eqref{Armijo} and \eqref{Wolfe} 
and $\beta_n := \beta_n^{\mathrm{PRP+}}$ (Theorem \ref{thm:4}).
\item[\textbf{HS+}:]
Algorithm \eqref{proposed} with $\alpha_n$ satisfying the Wolfe-type conditions \eqref{Armijo} and \eqref{Wolfe} 
and $\beta_n := \beta_n^{\mathrm{HS+}}$ (Theorem \ref{thm:5}).
\item[\textbf{DY}:]
Algorithm \eqref{proposed} with $\alpha_n$ satisfying the Wolfe-type conditions \eqref{Armijo} and \eqref{Wolfe} 
and $\beta_n := \beta_n^{\mathrm{DY}}$ (Theorem \ref{thm:2}).
\item[\textbf{HZ}:]
Algorithm \eqref{proposed} with $\alpha_n$ satisfying the Wolfe-type conditions \eqref{Armijo} and \eqref{Wolfe} 
and $\beta_n := \beta_n^{\mathrm{HZ}}$ defined by \eqref{HZ} \cite{HZ2005,HZ2006}.
\end{description}

The experiment used the following line search algorithm \cite[Algorithm 4.6]{lewis2013} to find step sizes satisfying 
the Wolfe-type conditions \eqref{Armijo} and \eqref{Wolfe} with $\delta := 0.3$ and $\sigma := 0.5$ 
that were chosen by referring to \cite[Subsection 6.1]{lewis2013}, 
where, for each $n$, $A_n(\cdot)$ and $W_n(\cdot)$ are 
\begin{align*}
A_n(t) &\colon \left\|x_n \left(t \right) - T \left(x_n\left(t\right) \right) \right\|^2 
- \left\|x_n  - T \left(x_n \right) \right\|^2 < \delta t \left\langle x_n  - T \left(x_n \right), d_n \right\rangle,\\
W_n(t) &\colon \left\langle x_n\left(t \right)  - T \left(x_n\left(t \right) \right), d_n \right\rangle 
> \sigma \left\langle x_n  - T \left(x_n \right), d_n \right\rangle.
\end{align*}

\begin{algo}{\em \cite[Algorithm 4.6]{lewis2013}}\label{line_search}
\begin{algorithmic}
{\em
\REQUIRE $A_n(\cdot), 
W_n(\cdot) 
$.
\ENSURE $A_n(\alpha)$ and $W_n(\alpha)$.
\STATE $\alpha\leftarrow{}0, \beta\leftarrow{}\infty, t\leftarrow{}1$.
\LOOP
\IF{$\lnot{}A_n(t)$}
\STATE $\beta\leftarrow{}t$.
\ELSIF{$\lnot{}W_n(t)$}
\STATE $\alpha\leftarrow{}t$.
\ELSE
\STATE (\textit{$\alpha:$ found}).
\ENDIF
\IF{$\beta<\infty$}
\STATE $t\leftarrow{}\frac{1}{2}(\alpha+\beta)$.
\ELSE
\STATE $t\leftarrow{}2\alpha$.
\ENDIF
\ENDLOOP
}
\end{algorithmic}
\end{algo}
For Algorithm SD-2, we replaced $A_n(\cdot)$ above by 
\begin{align*}
A_n(t) &\colon g_n\left(t\right) - g_n(0) < - D t \left\| x_n  - T \left(x_n \right) \right\|^2,
\end{align*}
where $D := \delta = 0.3$ and $g_n$ is defined as in \eqref{gn},
and deleted $W_n(\cdot)$ from the line search algorithm.
For Algorithms FR, PRP+, HS+, DY, and HZ, if the step sizes satisfying the Wolfe-type conditions \eqref{Armijo} and \eqref{Wolfe} were not computed by using Algorithm \ref{line_search},
the step sizes were computed by using Algorithm \ref{line_search} when $d_n := -(x_n - T(x_n))$.
This is because Algorithm \ref{line_search} for Algorithm SD-3, which uses $d_n := -(x_n - T(x_n))$ $(n\in\mathbb{N})$, 
had a 100\% success rate in computing the step sizes satisfying \eqref{Armijo} and \eqref{Wolfe}. 
Tables \ref{SD_0}, \ref{CG_0}, \ref{SD}, and \ref{CG} indicate the satisfiability rates (defined below) of computing the step sizes for the algorithms in the experiment.

The stopping condition was 
\begin{align}\label{stop}
n = 10 \text{ or } \left\| x_{n_0}  -T \left(x_{n_0}\right) \right\| = 0 \text{ for some } n_0 \in [0,10].
\end{align}
Before describing the results, let us describe the notation used to verify the numerical performance of the algorithms.
\begin{itemize} 
\item
$I \colon$ the number of initial points
\item
$x_0^{(i)} \colon$ the initial point chosen randomly $(i=1,2,\ldots, I)$
\item
$\mathrm{ALGO} \colon$ each of Algorithms SD-1, SD-2, SD-3, FR, PRP+, HS+, DY, and HZ 
($\mathrm{ALGO} \in \{$SD-1, SD-2, SD-3, FR, PRP+, HS+, DY, HZ$\}$)
\item
$N_1 (x_0^{(i)}, \mathrm{ALGO}) \colon$ the number of step sizes computed by Algorithm \ref{line_search} for ALGO with $x_0^{(i)}$
before ALGO satisfies the stopping condition \eqref{stop}
\item
$N_2 (x_0^{(i)}, \mathrm{ALGO}) \colon$ the number of iterations needed to satisfy the stopping condition \eqref{stop}
for ALGO with $x_0^{(i)}$
\end{itemize}
Note that $N_1 (x_0^{(i)},$ SD-1) stands for the number of iterations $n$ 
satisfying $A_n(0.5)$ and $W_n(0.5)$ before Algorithm SD-1 with $x_0^{(i)}$ satisfies the stopping condition \eqref{stop}.
The satisfiability rate (SR) of Algorithm \ref{line_search} to compute the step sizes for each of the algorithms is defined by
\begin{align}\label{SR}
\text{SR(ALGO)} := 
\frac{\sum_{i=1}^I N_1 \left(x_0^{(i)}, \mathrm{ALGO}\right)}
{\sum_{i=1}^I N_2 \left(x_0^{(i)}, \mathrm{ALGO} \right)} \times 100 \text{ } [\%].
\end{align}
We performed 100 samplings, each starting from different random initial points (i.e., $I := 100$) and averaged their results.

\subsection{Constrained quadratic programming problem}\label{subsec:3.2}
In this subsection, let us consider the following constrained quadratic programming problem:
\begin{prob}\label{QP}
Suppose that $C$ is a nonempty, closed convex subset of $\mathbb{R}^d$ onto which $P_C$ can be efficiently computed, 
$Q \in \mathbb{R}^{d \times d}$ is positive semidefinite with 
the eigenvalues $\lambda_{\min} := \lambda_1, \lambda_2, \ldots, \lambda_d =: \lambda_{\max}$ satisfying $\lambda_i \leq \lambda_j$
($i \leq j$),
and $b\in \mathbb{R}^d$.
Our objective is to 
\begin{align*}
\text{minimize } f(x) := \frac{1}{2} \left\langle x,Qx \right\rangle + \left\langle b,x \right\rangle
\text{ subject to } x \in C.
\end{align*} 
\end{prob}
Since $f$ above is convex and $\nabla f(x) = Qx +b$ $(x\in \mathbb{R}^d)$ is Lipschitz continuous such that
the Lipschitz constant of $\nabla f$ is the maximum eigenvalue $\lambda_{\max}$ of $Q$, 
Problem \ref{QP} is an example of Problem \eqref{csco}.

We compared the proposed algorithms SD-3, FR, PRP+, HS+, DY, and HZ with the previous algorithms SD-1 and SD-2
by applying them to Problem \ref{QP} (i.e., the fixed point problem for $T(x) := P_C (x - (2/\lambda_{\max}) (Qx + b))$ $(x\in \mathbb{R}^d)$) 
in the following cases:
\begin{align*}
&d:= 10^3 \text{ or } 10^4, \text{ } \lambda_{\min} := 0, \text{ } \lambda_{\max} := d, \text{ } 
\lambda_i \in [0,d] \text{ } (i=2,3,\dots, d-1),\\
&b, c \in (-32, 32)^d, \text{ } 
C := \left\{ x \in \mathbb{R}^d \colon \left\| x - c \right\| \leq 1  \right\}.
\end{align*}  
We randomly chose $\lambda_i \in [0,d]$ $(i=2,3,\ldots,d-1)$ and set $Q$ as a diagonal
matrix with eigenvalues $\lambda_1, \lambda_2, \ldots, \lambda_{\max}$. 
The experiment used two random numbers in the range of $(-32,32)^d$ for $b$ and $c$ to satisfy 
$C \cap \{ x\in \mathbb{R}^d \colon \nabla f(x) = 0 \} = \emptyset$.
Since $C$ is a closed ball with center $c$ and radius $1$, $P_C$ can be computed within a finite number of arithmetic operations. More precisely, $P_C (x) := c + (x- c)/\| x- c \|$ if $\| x - c \| > 1$, or $P_C (x) := x$ if $\| x - c \| \leq 1$. 

\begin{table}[htbp]
  \caption{Satisfiability rate of Algorithm \ref{line_search} for Algorithms SD-1, SD-2, and SD-3 applied to Problem \ref{QP} 
  when $d:= 10^3, 10^4$
  }\label{SD_0}
  \centering
  \begin{tabular}{c|cc}
    \hline
      Algorithm & SR $(d:=10^3)$ & SR ($d:=10^4$)  \\
    \hline
      SD-1      & 55.9\%          & 26.3\%  \\
      SD-2      & 100\%           & 100\%   \\
      SD-3      & 100\%           & 100\%   \\
    \hline
  \end{tabular}
\end{table}
\begin{table}[htbp]
  \caption{Satisfiability rate of Algorithm \ref{line_search} for Algorithms FR, PRP+, HS+, DY, and HZ applied to Problem \ref{QP} 
  when $d:= 10^3, 10^4$}\label{CG_0}
  \centering
  \begin{tabular}{c|cc}
    \hline
      Algorithm & SR ($d:=10^3$) & SR ($d := 10^4$) \\
    \hline
      FR & 19.7\% & 28.1\% \\
      PRP+ & 100\% & 100\% \\
      HS+ & 100\% & 98.9\% \\
      DY & 21.6\% & 27.2\% \\
      HZ & 20.0\% & 20.0\% \\
    \hline
  \end{tabular}
\end{table}

Table \ref{SD_0} shows the satisfiability rates as defined by \eqref{SR} for Algorithms SD-1, SD-2, and SD-3 
that are applied to Problem \ref{QP}. 
It can be seen that the step sizes for SD-1 (constant step sizes $\alpha_n := 0.5$) do not always satisfy the Wolfe-type conditions \eqref{Armijo} and \eqref{Wolfe}, whereas the step sizes computed by Algorithm \ref{line_search} and SD-2 (resp. Algorithm SD-3) definitely satisfy the Armijo-type condition \eqref{armijo0} (resp. the Wolfe-type conditions \eqref{Armijo} and \eqref{Wolfe}). 

Table \ref{CG_0} showing the satisfiability rates for Algorithms FR, PRP+, HS+, DY, and HZ
indicates that Algorithm \ref{line_search} for PRP+ and HS+ has high success rates at computing the step sizes satisfying \eqref{Armijo} and \eqref{Wolfe}, while the SRs of Algorithm \ref{line_search} for other algorithms are low. 
It can be seen from Tables \ref{SD_0} and \ref{CG_0} that SD-3, PRP+, and HS+ are robust in the sense that Algorithm \ref{line_search} can compute the step sizes satisfying the Wolfe-type conditions \eqref{Armijo} and \eqref{Wolfe}.  

Figure \ref{sd_qp} indicates the behaviors of SD-1, SD-2, and SD-3 when $d:= 10^3$. The y-axes in Figures \ref{sd_qp}(a) and \ref{sd_qp}(b) represent the value of $\|x_n - T(x_n)\|$. The x-axis in Figure \ref{sd_qp}(a) represents the number of iterations, and the x-axis in Figure \ref{sd_qp}(b) represents the elapsed time. If the $(\| x_n - T(x_n) \|)_{n\in\mathbb{N}}$ generated by the algorithms converges to $0$, they also converge to a fixed point of $T$. Figure \ref{sd_qp}(a) shows that SD-2 and SD-3 terminate at fixed points of $T$ within a finite number of iterations. It can be seen from Figure \ref{sd_qp}(a) and Figure \ref{sd_qp}(b) that SD-3 reduces the iterations and running time needed to find a fixed point compared with SD-2. These figures also show that $(\|x_n - T(x_n)\|)_{n\in\mathbb{N}}$ generated by SD-1 converges slowest and that SD-1 cannot find a fixed point of $T$ before the tenth iteration. We can thus see that the use of the step sizes satisfying the Wolfe-type conditions is a good way to solve fixed point problems by using the Krasnosel'ski\u\i-Mann algorithm. Figure \ref{sd_qp1} indicates the behaviors of SD-1, SD-2, and SD-3 when $d:= 10^4$. Similarly to what is shown in Figure \ref{sd_qp}, SD-3 finds a fixed point of $T$ faster than SD-1 and SD-2 can.

\begin{figure}[H]
  \centering
  \subcaptionbox{$\|x_n - T(x_n) \|$ vs. no. of iterations}{\includegraphics{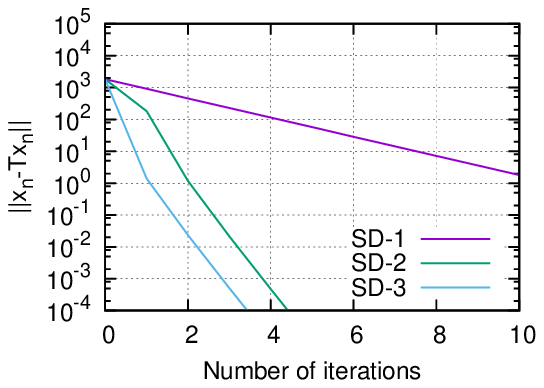}}
  \subcaptionbox{$\|x_n - T(x_n) \|$ vs. elapsed time}{\includegraphics{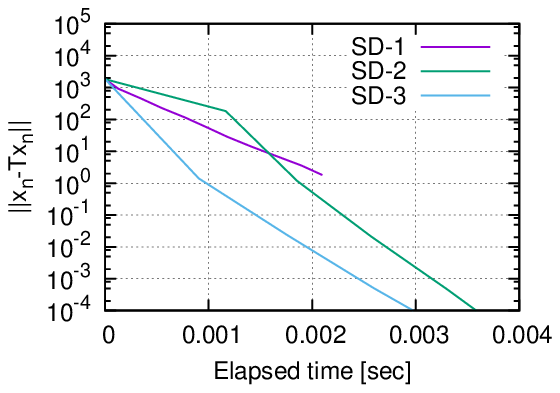}}
  \caption{Evaluation of $\|x_n - T(x_n)\|$ in terms of the number of iterations and elapsed time for Algorithms SD-1, SD-2, and SD-3
  for Problem \ref{QP} when $d := 10^3$}\label{sd_qp}
\end{figure}

\begin{figure}[H]
  \centering
  \subcaptionbox{$\|x_n - T(x_n) \|$ vs. no. of iterations}{\includegraphics{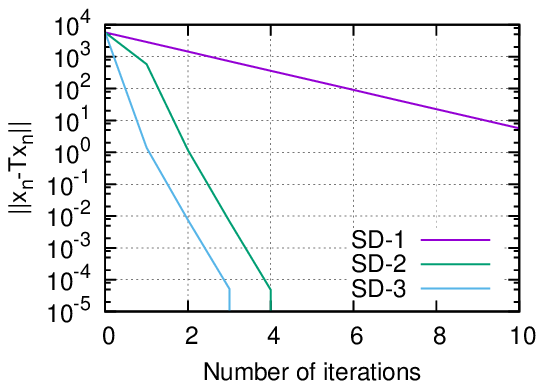}}
  \subcaptionbox{$\|x_n - T(x_n) \|$ vs. elapsed time}{\includegraphics{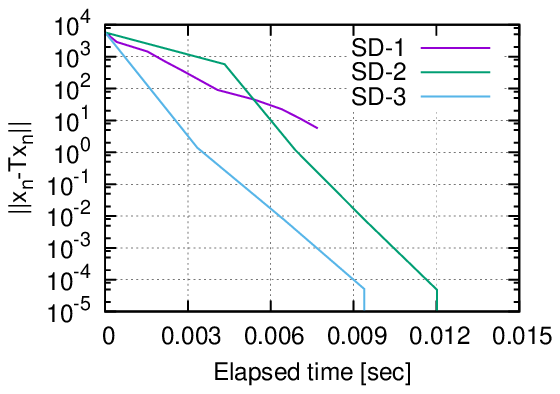}}
  \caption{Evaluation of $\|x_n - T(x_n)\|$ in terms of the number of iterations and elapsed time for Algorithms SD-1, SD-2, and SD-3
  for Problem \ref{QP} when $d := 10^4$}\label{sd_qp1}
\end{figure}

Figure \ref{cg_qp} is the evaluation of $(\| x_n - T(x_n) \|)_{n\in\mathbb{N}}$ in terms of the number of iterations 
and elapsed time for Algorithms FR, PRP+, HS+, DY, and HZ when $d:= 10^3$. 
Figure \ref{cg_qp}(a) shows that they can find fixed points of $T$ within a finite number of iterations.
Figure \ref{cg_qp}(b) indicates that PRP+ and HS+ find the fixed points of $T$ faster than FR, DY, and HZ.
This is because Algorithm \ref{line_search} for each of PRP+ and HS+ has a 100 \% success rate at computing the step sizes
satisfying \eqref{Armijo} and \eqref{Wolfe}, while the SRs of Algorithm \ref{line_search} for FR, DY, and HZ 
are low (see Table \ref{CG_0}); i.e., FR, DY, and HZ require much more time to compute the step sizes than PRP+ and HS+. 
In fact, we checked that the times to compute the step sizes for FR, DY, and HZ account for 92.672202\%, 87.156303\%, and 83.700936\% of all the computational times, while the times to compute the step sizes for PRP+ and HS+ account for 60.725204\% and 60.889635\% of all the computational times.
Figure \ref{cg_qp1} indicate the behaviors of FR, PRP+, HS+, DY, and HZ when $d:= 10^4$ and 
PRP+ and HS+ perform better than FR, DY, and HZ, as seen in Figure \ref{cg_qp}.
Such a trend can be also verified from Table \ref{CG_0} showing that the SRs of Algorithm \ref{line_search} 
for PRP+ and HS+ are about 100\%.

\begin{figure}[H]
  \centering
  \subcaptionbox{$\|x_n - T(x_n) \|$ vs. no. of iterations}{\includegraphics{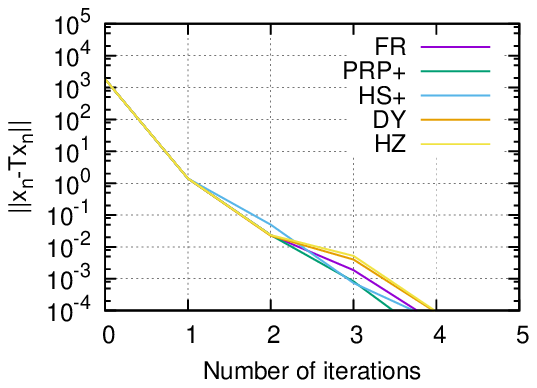}}
  \subcaptionbox{$\|x_n - T(x_n) \|$ vs. elapsed time}{\includegraphics{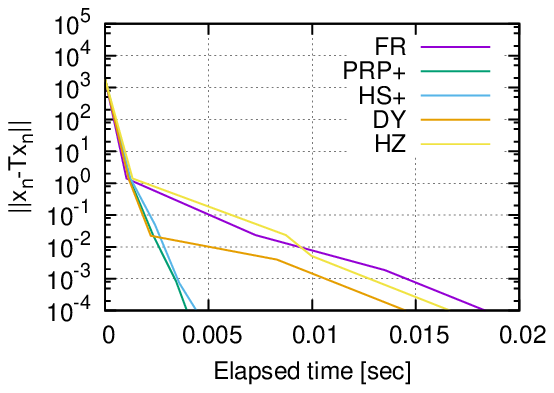}}
  \caption{Evaluation of $\|x_n - T(x_n)\|$ in terms of the number of iterations and elapsed time for Algorithms
  FR, PRP+, HS+, DY, and HZ
  for Problem \ref{QP} when $d := 10^3$}\label{cg_qp}
\end{figure}

\begin{figure}[H]
  \centering
  \subcaptionbox{$\|x_n - T(x_n) \|$ vs. no. of iterations}{\includegraphics{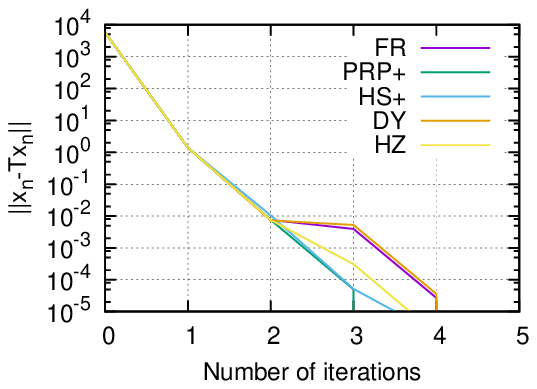}}
  \subcaptionbox{$\|x_n - T(x_n) \|$ vs. elapsed time}{\includegraphics{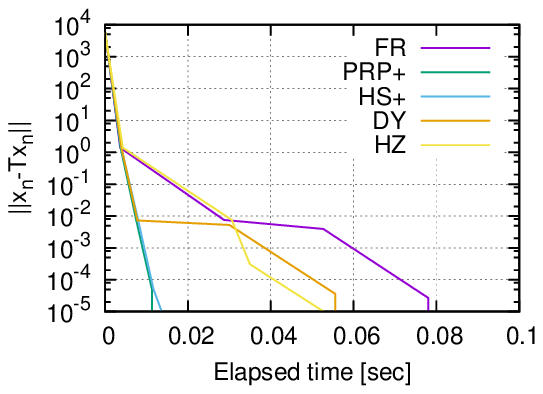}}
  \caption{Evaluation of $\|x_n - T(x_n)\|$ in terms of the number of iterations and elapsed time for Algorithms
  FR, PRP+, HS+, DY, and HZ
  for Problem \ref{QP} when $d := 10^4$}\label{cg_qp1}
\end{figure}

\subsection{Generalized convex feasibility problem}\label{subsec:3.1}
This subsection considers the following generalized convex feasibility problem \cite[Section I, Framework 2]{com1999},
\cite[Subsection 2.2]{iiduka_siam2012}, \cite[Definition 4.1]{yamada}:

\begin{prob}\label{GCFP}
Suppose that $C_i$ $(i=0,1,\ldots, m)$ is a nonempty, closed convex subset of $\mathbb{R}^d$ 
onto which $P_{C_i}$ can be efficiently computed
and define the weighted mean square value of the distances from $x\in \mathbb{R}^d$ to $C_i$ $(i=1,2,\ldots,m)$ as $f(x)$ below; i.e., for $w_i\in (0,1)$ $(i=1,2,\ldots,m)$ satisfying $\sum_{i=1}^m w_i = 1$,
\begin{align*}
f\left(x\right) 
:=  \sum_{i=1}^m w_i \left( \min_{y\in C_i} \left\| x - y  \right\| \right)^2.
\end{align*}
Our objective is to find a point in the generalized convex feasible set defined by 
\begin{align*}
C_{f} := \left\{ x^\star \in C_0 \colon f \left( x^\star \right) = \min_{x\in C_0} f \left(x\right)  \right\}.
\end{align*}
\end{prob}

$C_{f}$ is a subset of $C_0$ having the elements closest to $C_i$ $(i=1,2,\ldots,m)$ in terms of the weighted mean square norm. 
Even if $\bigcap_{i=0}^m C_i = \emptyset$, $C_{f}$ is well-defined because $C_{f}$ is the set of all minimizers of $f$ over $C_0$. 
The condition $C_{f} \neq \emptyset$ holds when $C_0$ is bounded \cite[Remark 4.3(a)]{yamada}. 
Moreover, $C_{f} = \bigcap_{i=0}^m C_i$ holds when $\bigcap_{i=0}^m C_i \neq \emptyset$. 
Accordingly, Problem \ref{GCFP} is a generalization of the convex feasibility problem \cite{bau} 
of finding a point in $\bigcap_{i=0}^m C_i \neq \emptyset$.

The convex function $f$ in Problem \ref{GCFP} satisfies $\nabla f = \mathrm{Id} - \sum_{i=1}^m w_i P_{C_i}$.
Hence, $\nabla f$ is Lipschitz continuous when its Lipschitz constant is two.
This means Problem \ref{GCFP} is an example of Problem \eqref{csco}.
Since Problem \ref{GCFP} can be expressed as the problem of finding a fixed point of  
$T = P_{C_0} (\mathrm{Id} - \lambda \nabla f) = P_{C_0} (\mathrm{Id} - \lambda (\mathrm{Id} - \sum_{i=1}^m w_i P_{C_i}) )$
for $\lambda \in (0,1]$,
we used $T$ with $\lambda =1$; i.e., $T := P_{C_0} (\sum_{i=1}^m w_i P_{C_i})$.

We applied SD-1, SD-2, SD-3, FR, PRP+, HS+, DY, and HZ to Problem \ref{GCFP} in the following cases:
\begin{align*}
&d:= 10^3 \text{ or } 10^4, \text{ } 
m:= 99, \text{ } w_i := \frac{1}{99} \text{ } (i=1,2,\ldots,99),\\
&c_i \in (-32, 32)^d, \text{ } 
C_i := \left\{ x \in \mathbb{R}^d \colon \left\| x - c_i \right\| \leq 1  \right\} \text{ } (i=0,1,\ldots,m).
\end{align*}
The experiment used one hundred random numbers in the range of $(-32,32)^d$ for $c_i$, which means $\bigcap_{i=0}^m C_i = \emptyset$. Since $C_i$ $(i=0,1,\ldots,m)$ is a closed ball with center $c_i$ and radius $1$, $P_{i}$ can be computed within a finite number of arithmetic operations. 

\begin{table}[htbp]
  \caption{Satisfiability rate of Algorithm \ref{line_search} for Algorithms SD-1, SD-2, and SD-3 applied to Problem \ref{GCFP} 
  when $d:= 10^3, 10^4$
  }\label{SD}
  \centering
  \begin{tabular}{c|cc}
    \hline
      Algorithm & SR $(d:=10^3)$ & SR ($d:=10^4$)  \\
    \hline
      SD-1      & 80.6\%          & 64.2\%  \\
      SD-2      & 100\%           & 100\%   \\
      SD-3      & 100\%           & 100\%   \\
    \hline
  \end{tabular}
\end{table}

\begin{table}[htbp]
  \caption{Satisfiability rate of Algorithm \ref{line_search} for Algorithms FR, PRP+, HS+, DY, and HZ applied to Problem \ref{GCFP}
  when $d:= 10^3, 10^4$}\label{CG}
  \centering
  \begin{tabular}{c|cc}
    \hline
      Algorithm & SR ($d:=10^3$) & SR ($d := 10^4$) \\
    \hline
      FR & 50.0\% & 50.0\% \\
      PRP+ & 100\% & 100\% \\
      HS+ & 55.8\% & 60.4\% \\
      DY & 50.0\% & 50.0\% \\
      HZ & 50.0\% & 50.0\% \\
    \hline
  \end{tabular}
\end{table}

Table \ref{SD} shows the satisfiability rates as defined by \eqref{SR} for Algorithms SD-1, SD-2, and SD-3 applied to Problem \ref{GCFP}. It can be seen that the step sizes for SD-1 do not always satisfy the Wolfe-type conditions \eqref{Armijo} and \eqref{Wolfe}, whereas the step sizes computed by Algorithm \ref{line_search} and SD-2 (resp. Algorithm SD-3) definitely satisfy the Armijo-type condition \eqref{armijo0} (resp. the Wolfe-type conditions \eqref{Armijo} and \eqref{Wolfe}). Such a trend also existed when SD-1, SD-2, and SD-3 were applied to Problem \ref{QP} (see Table \ref{SD_0}).

Table \ref{CG} shows the satisfiability rates for Algorithms FR, PRP+, HS+, DY, and HZ. The table indicates that Algorithm \ref{line_search} for PRP+ has a 100\% success rate at computing the step sizes satisfying \eqref{Armijo} and \eqref{Wolfe}, while the SRs of Algorithm \ref{line_search} for the other algorithms lie between $50\%$ and about $60\%$. From Tables \ref{SD} and \ref{CG}, we can see that SD-3 and PRP+ are robust in the sense that Algorithm \ref{line_search} can compute the step sizes satisfying the Wolfe-type conditions \eqref{Armijo} and \eqref{Wolfe}.

\begin{figure}[H]
  \centering
  \subcaptionbox{$\|x_n - T(x_n) \|$ vs. no. of iterations}{\includegraphics{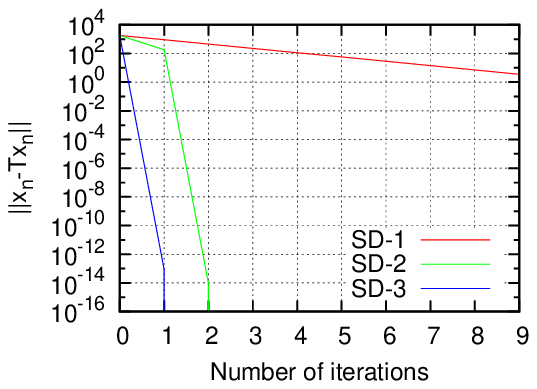}}
  \subcaptionbox{$\|x_n - T(x_n) \|$ vs. elapsed time}{\includegraphics{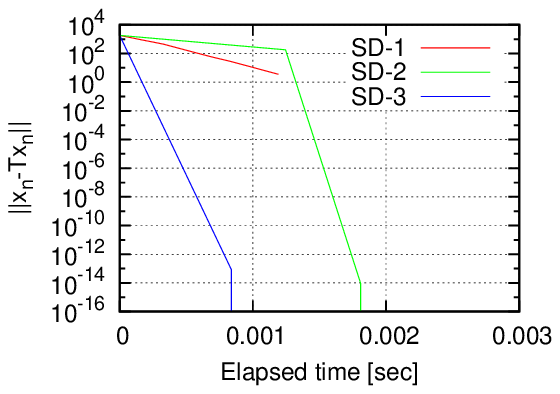}}
  \caption{Evaluation of $\|x_n - T(x_n)\|$ in terms of the number of iterations and elapsed time for Algorithms SD-1, SD-2, and SD-3
  for Problem \ref{GCFP} when $d := 10^3$}\label{fig:1}
\end{figure}

\begin{figure}[H]
  \centering
  \subcaptionbox{$\|x_n - T(x_n) \|$ vs. no. of iterations}{\includegraphics{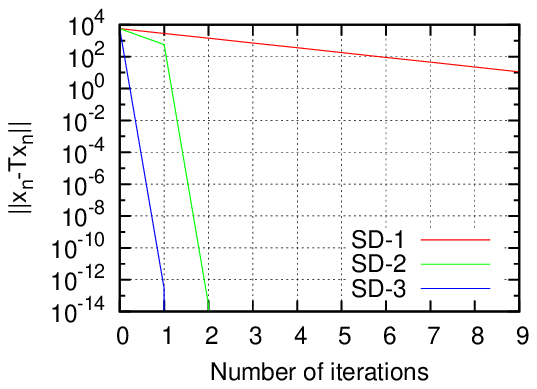}}
  \subcaptionbox{$\|x_n - T(x_n) \|$ vs. elapsed time}{\includegraphics{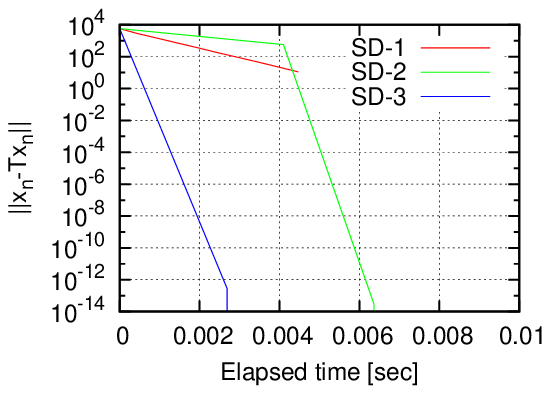}}
  \caption{Evaluation of $\|x_n - T(x_n)\|$ in terms of the number of iterations and elapsed time for Algorithms SD-1, SD-2, and SD-3
  for Problem \ref{GCFP} when $d := 10^4$}\label{fig:3}
\end{figure}

Figure \ref{fig:1} indicates the behaviors of SD-1, SD-2, and SD-3 when $d:= 10^3$. The y-axes represent the value of $\| x_n - T(x_n) \|$. The x-axis in Figure \ref{fig:1}(a) represents the number of iterations, and the x-axis in Figure \ref{fig:1}(b) represents the elapsed time. From Figure \ref{fig:1}(a), the iterations needed to satisfy $\|x_n - T(x_n) \| = 0$ for SD-2 and SD-3 are, respectively, $3$ and $2$. It can be seen that SD-3 reduces the running time and iterations needed to find a fixed point compared with SD-2. These figures also show that the $(\|x_n - T(x_n)\|)_{n\in\mathbb{N}}$ generated by SD-1 converges slowest. Therefore, we can see that the use of the step sizes satisfying the Wolfe-type conditions is a good way to solve fixed point problems by using the Krasnosel'ski\u\i-Mann algorithm, as seen in Figures \ref{sd_qp} and \ref{sd_qp1} illustrating the behaviors of SD-1, SD-2, and SD-3 on Problem \ref{QP} when $d := 10^3, 10^4$. Figure \ref{fig:3} indicates the behaviors of SD-1, SD-2, and SD-3 when $d:= 10^4$. Similarly to what is shown in Figure \ref{fig:1}, SD-3 finds a fixed point of $T$ faster than SD-1 and SD-2 can.

\begin{figure}[H]
  \centering
  \subcaptionbox{$\|x_n - T(x_n) \|$ vs. no. of iterations}{\includegraphics{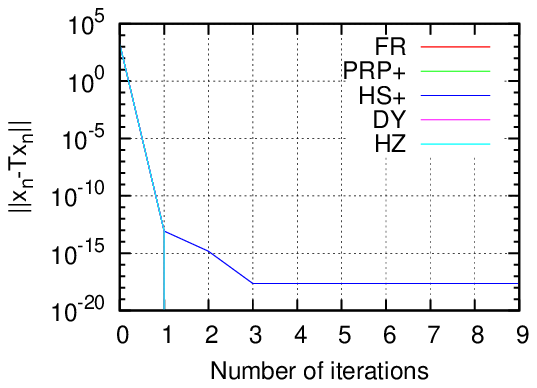}}
  \subcaptionbox{$\|x_n - T(x_n) \|$ vs. elapsed time}{\includegraphics{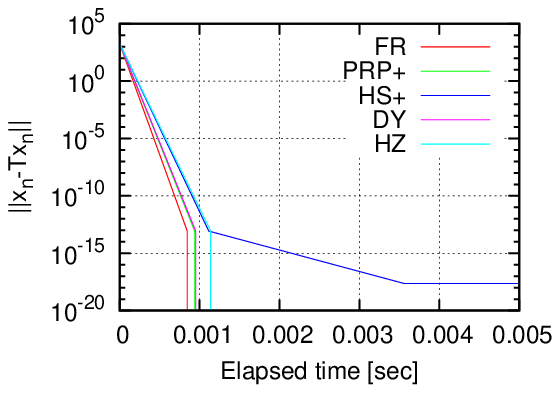}}
  \caption{Evaluation of $\|x_n - T(x_n)\|$ in terms of the number of iterations and elapsed time for Algorithms FR, PRP+, 
  HS+, DY, and HZ
  for Problem \ref{GCFP} when $d := 10^3$}\label{fig:5}
\end{figure}

\begin{figure}[H]
  \centering
  \subcaptionbox{$\|x_n - T(x_n) \|$ vs. no. of iterations}{\includegraphics{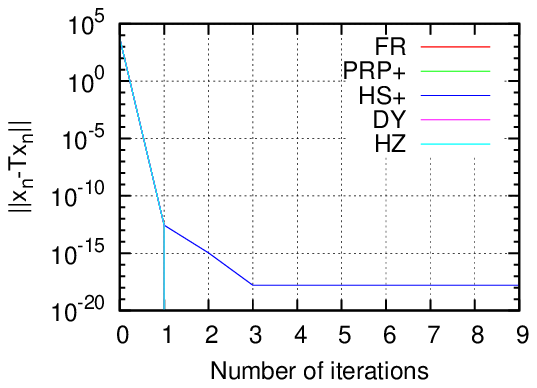}}
  \subcaptionbox{$\|x_n - T(x_n) \|$ vs. elapsed time}{\includegraphics{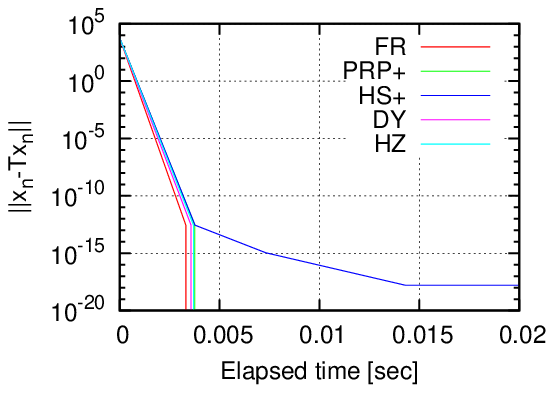}}
  \caption{Evaluation of $\|x_n - T(x_n)\|$ in terms of the number of iterations and elapsed time for Algorithms FR, PRP+, 
  HS+, DY, and HZ
  for Problem \ref{GCFP} when $d := 10^4$}\label{fig:7}
\end{figure}

Figure \ref{fig:5}(a) is the evaluation of $(\| x_n - T(x_n) \|)_{n\in\mathbb{N}}$ in terms of the number of iterations for Algorithms FR, PRP+, HS+, DY, and HZ when $d:= 10^3$. Except for HS+, the algorithms approximate the fixed points of $T$ very rapidly.
It can be also seen that the algorithms other than HS+ satisfy $\| x_2 - T(x_2) \| = 0$. 
Figure \ref{fig:5}(b) is the evaluation of $(\| x_n - T(x_n) \|)_{n\in\mathbb{N}}$ in terms of the elapsed time. 
Here, we can see that FR, PRP+, and DY can find fixed points of $T$ faster than SD-1 and SD-2 (Figure \ref{fig:1}). 
Figure \ref{fig:7} indicates the behaviors of FR, PRP+, HS+, DY, and HZ when $d:= 10^4$. The results in these figures are almost the same as the ones in Figures \ref{fig:5}.

From the above numerical results, we can conclude that the proposed algorithms can find optimal solutions to Problems \ref{QP} and \ref{GCFP} faster than the previous fixed point algorithms can. In particular, it can be seen that the algorithms for which the SRs of Algorithm \ref{line_search} are high converge quickly to solutions of Problems \ref{QP} and \ref{GCFP}.

\section{Conclusion and Future Work}\label{sec:4}
This paper discussed the fixed point problem for a nonexpansive mapping on a real Hilbert space and presented line search fixed point algorithms for solving it on the basis of nonlinear conjugate gradient methods for unconstrained optimization and 
their convergence analyses and convergence rate analyses. 
Moreover, we used these algorithms to solve concrete constrained quadratic programming problems and generalized convex feasibility problems and numerically compared them with the previous fixed point algorithms based on the Krasnosel'ski\u\i-Mann fixed point algorithm. 
The numerical results showed that the proposed algorithms can find optimal solutions to these problems faster than the previous algorithms.

In the experiment, the line search algorithm (Algorithm \ref{line_search}) could not compute appropriate step sizes for fixed point algorithms other than Algorithms SD-2, SD-3, and PRP+. In the future, we should consider modifying the algorithms to enable the line search to compute appropriate step sizes. Or we may need to develop new line searches that can be applied to all of the fixed point algorithms considered in this paper.

The main objective of this paper was to devise line-search fixed-point algorithms to accelerate the previous Krasnosel'ski\u\i-Mann fixed point algorithm defined by \eqref{kra}, i.e., $x_{n+1} := \lambda_n x_n + (1-\lambda_n) T(x_n)$ $(n\in \mathbb{N})$, where $(\lambda_n)_{n\in\mathbb{N}} \subset [0,1]$ with $\sum_{n=0}^\infty \lambda_n (1-\lambda_n) = \infty$ and $x_0 \in H$ is an initial point. Another particularly interesting problem is determining whether or not there are line search fixed point algorithms to accelerate the following Halpern fixed point algorithm \cite{halpern,wit}: for all $n\in \mathbb{N}$,
\begin{align*}
x_{n+1} &:= \alpha_n x_0 + \left(1-\alpha_n \right) T \left(x_n \right),
\end{align*}
where $(\alpha_n)_{n\in\mathbb{N}} \subset (0,1)$ satisfies $\lim_{n\to\infty} \alpha_n = 0$ and $\sum_{n=0}^\infty \alpha_n = \infty$. The Halpern algorithm can minimize the convex function $\| \cdot - x_0 \|^2$ over $\mathrm{Fix}(T)$ (see, e.g., \cite[Theorem 6.17]{berinde}). A previously reported result \cite[Theorem 3.1, Proposition 3.2]{iiduka_mp2014} showed that there is an inconvenient possibility that the Halpern-type algorithm with a diminishing step size sequence (e.g., $\alpha_n := 1/(n+1)^a$, where $a \in (0,1]$) and any of the FR, PRP, HS, and DY formulas used in the conventional conjugate gradient methods may not converge to the minimizer of $\| \cdot - x_0 \|^2$ over $\mathrm{Fix}(T)$. However, there is room for further research into devising line search fixed point algorithms to accelerate the Halpern algorithm with a diminishing step size sequence.

\section*{Acknowledgments} 
The author thanks Mr. Kazuhiro Hishinuma for his discussion on the numerical experiments.

\bibliographystyle{spmpsci} 
\bibliography{biblio1}

\begin{thebibliography}{10}
\providecommand{\url}[1]{{#1}}
\providecommand{\urlprefix}{URL }
\expandafter\ifx\csname urlstyle\endcsname\relax
  \providecommand{\doi}[1]{DOI~\discretionary{}{}{}#1}\else
  \providecommand{\doi}{DOI~\discretionary{}{}{}\begingroup
  \urlstyle{rm}\Url}\fi

\bibitem{Al-Baali1985}
Al-Baali, M.: Descent property and global convergence of the
  {F}letcher--{R}eeves method with inexact line search.
\newblock IMA Journal of Numerical Analysis \textbf{5}, 121--124 (1985)

\bibitem{bau}
Bauschke, H.H., Borwein, J.M.: On projection algorithms for solving convex
  feasibility problems.
\newblock SIAM Review \textbf{38}(3), 367--426 (1996)

\bibitem{b-c}
Bauschke, H.H., Combettes, P.L.: Convex Analysis and Monotone Operator Theory
  in Hilbert Spaces.
\newblock Springer (2011)

\bibitem{berinde}
Berinde, V.: Iterative Approximation of Fixed Points.
\newblock Springer (2007)

\bibitem{bot2015}
Bo\c{t}, R.I., Csetnek, E.R.: A dynamical system associated with the fixed
  points set of a nonexpansive operator.
\newblock Journal of Dynamics and Differential Equations (to appear) (2015)

\bibitem{com1999}
Combettes, P.L., Bondon, P.: Hard-constrained inconsistent signal feasibility
  problems.
\newblock IEEE Transactions on Signal Processing \textbf{47}(9), 2460--2468
  (1999)

\bibitem{comb2007}
Combettes, P.L., Pesquet, J.C.: A {D}ouglas-{R}achford splitting approach to
  nonsmooth convex variational signal recovery.
\newblock IEEE Journal of Selected Topics in Signal Processing \textbf{1},
  564--574 (2007)

\bibitem{cominetti2014}
Cominetti, R., Soto, J.A., Vaisman, J.: On the rate of convergence of
  {K}rasnosel'ski\u\i-{M}ann iterations and their connection with sums of
  {B}ernoullis.
\newblock Israel Journal of Mathematics \textbf{199}, 757--772 (2014)

\bibitem{DY1999}
Dai, Y.H., Yuan, Y.: A nonlinear conjugate gradient method with a strong global
  convergence property.
\newblock SIAM Journal on Optimization \textbf{10}, 177--182 (1999)

\bibitem{FR1964}
Fletcher, R., Reeves, C.: Function minimization by conjugate gradients.
\newblock Computer Journal \textbf{7}, 149--154 (1964)

\bibitem{gilbert1992}
Gilbert, J.C., Nocedal, J.: Global convergence properties of conjugate gradient
  methods for optimization.
\newblock SIAM Journal on Optimization \textbf{2}, 21--42 (1992)

\bibitem{goebel1}
Goebel, K., Kirk, W.A.: Topics in Metric Fixed Point Theory.
\newblock Cambridge Studies in Advanced Mathematics. Cambridge University Press
  (1990)

\bibitem{goebel2}
Goebel, K., Reich, S.: Uniform Convexity, Hyperbolic Geometry, and Nonexpansive
  Mappings.
\newblock Dekker (1984)

\bibitem{HZ2005}
Hager, W.W., Zhang, H.: A new conjugate gradient method with guaranteed descent
  and an efficient line search.
\newblock SIAM Journal on Optimization \textbf{16}, 170--192 (2005)

\bibitem{HZ2006}
Hager, W.W., Zhang, H.: Algorithm 851: {CG}\_{DESCENT}: a conjugate gradient
  method with guaranteed descent.
\newblock ACM Transactions on Mathematical Software \textbf{32}, 113--137
  (2006)

\bibitem{hager2006}
Hager, W.W., Zhang, H.: A survey of nonlinear conjugate gradient methods.
\newblock Pacific Journal of Optimization \textbf{2}, 35--58 (2006)

\bibitem{halpern}
Halpern, B.: Fixed points of nonexpanding maps.
\newblock Bulletin of the American Mathematical Society \textbf{73}, 957--961
  (1967)

\bibitem{HS1952}
Hestenes, M.R., Stiefel, E.L.: Methods of conjugate gradients for solving
  linear systems.
\newblock Journal of Research of the National Bureau of Standards \textbf{49},
  409--436 (1952)

\bibitem{iiduka_JOTA}
Iiduka, H.: Iterative algorithm for solving triple-hierarchical constrained
  optimization problem.
\newblock Journal of Optimization Theory and Applications \textbf{148},
  580--592 (2011)

\bibitem{iiduka_siam2012}
Iiduka, H.: Iterative algorithm for triple-hierarchical constrained nonconvex
  optimization problem and its application to network bandwidth allocation.
\newblock SIAM Journal on Optimization \textbf{22}(3), 862--878 (2012)

\bibitem{iiduka_mp2014}
Iiduka, H.: Acceleration method for convex optimization over the fixed point
  set of a nonexpansive mapping.
\newblock Mathematical Programming \textbf{149}, 131--165 (2015)

\bibitem{kra}
Krasnosel'ski\u\i, M.A.: Two remarks on the method of successive
  approximations.
\newblock Uspekhi Matematicheskikh Nauk \textbf{10}, 123--127 (1955)

\bibitem{lewis2013}
Lewis, A.S., Overton, M.L.: Nonsmooth optimization via quasi-{N}ewton methods.
\newblock Mathematical Programming \textbf{141}, 135--163 (2013)

\bibitem{mag2004}
Magnanti, T.L., Perakis, G.: Solving variational inequality and fixed point
  problems by line searches and potential optimization.
\newblock Mathematical Programming \textbf{101}, 435--461 (2004)

\bibitem{mann}
Mann, W.R.: Mean value methods in iteration.
\newblock Proceedings of American Mathematical Society \textbf{4}, 506--510
  (1953)

\bibitem{nakajo2003}
Nakajo, K., Takahashi, W.: Strong convergence theorems for nonexpansive
  mappings and nonexpansive semigroups.
\newblock Journal of Mathematical Analysis and Applications \textbf{279},
  372--379 (2003)

\bibitem{noce}
Nocedal, J., Wright, S.J.: Numerical Optimization, 2nd edn.
\newblock Springer Series in Operations Research and Financial Engineering.
  Springer (2006)

\bibitem{opial}
Opial, Z.: Weak convergence of the sequence of successive approximation for
  nonexpansive mappings.
\newblock Bulletin of the American Mathematical Society \textbf{73}, 591--597
  (1967)

\bibitem{PR1969}
Polak, E., Ribi\`ere, G.: Note sur la convergence de directions conjug\'ees.
\newblock Revue Fran\c{c}aise d'automatique, Informatique, Recherche
  Op\'erationnelle \textbf{3e Ann\'ee 16}, 35--43 (1969)

\bibitem{P1969}
Polyak, B.T.: The conjugate gradient method in extreme problems.
\newblock USSR Computational Mathematics and Mathematical Physics \textbf{9},
  94--112 (1969)

\bibitem{powell1984}
Powell, M.J.D.: Nonconvex minimization calculations and the conjugate gradient
  method, vol. 1066, chap. Numerical Analysis (Dundee, 1983), Lecture Notes in
  Mathematics, pp. 122--141.
\newblock Springer-Verlag, Berlin (1984)

\bibitem{solodov2000}
Solodov, M.V., Svaiter, B.F.: Forcing strong convergence of proximal point
  iterations in a {H}ilbert space.
\newblock Mathematical Programming \textbf{87}, 189--202 (2000)

\bibitem{takahashi}
Takahashi, W.: Nonlinear Functional Analysis.
\newblock Yokohama Publishers (2000)

\bibitem{wit}
Wittmann, R.: Approximation of fixed points of nonexpansive mappings.
\newblock Archiv der Mathematik \textbf{58}(5), 486--491 (1992)

\bibitem{wolfe1969}
Wolfe, P.: Convergence conditions for ascent methods.
\newblock SIAM Review \textbf{11}, 226--235 (1969)

\bibitem{wolfe1971}
Wolfe, P.: Convergence conditions for ascent methods. {II}: {S}ome corrections.
\newblock SIAM Review \textbf{13}, 185--188 (1971)

\bibitem{yamada}
Yamada, I.: The hybrid steepest descent method for the variational inequality
  problem over the intersection of fixed point sets of nonexpansive mappings.
\newblock In: D.~Butnariu, Y.~Censor, S.~Reich (eds.) Inherently Parallel
  Algorithms for Feasibility and Optimization and Their Applications, pp.
  473--504. Elsevier (2001)

\bibitem{zou1970}
Zoutendijk, G.: Nonlinear programming, computational methods.
\newblock In: J.~Abadie (ed.) Integer and Nonlinear Programming, pp. 37--38.
  North-Holland, Amsterdam (1970)

\end{thebibliography}

\end{document}